\documentclass[11pt]{amsart}
\usepackage{amsmath,amsthm,amssymb}
\usepackage{fullpage}
\usepackage{hyperref}
\usepackage{verbatim}
\usepackage{xcolor}
\usepackage[all]{xy}
  \SelectTips{cm}{10}
  \everyxy={<2.5em,0em>:} 
\usepackage{tikz}
\usepackage{picinpar} 

\renewcommand{\theequation}{\thesection.\arabic{equation}}
\makeatletter\@addtoreset{equation}{section}\makeatother

\theoremstyle{plain}
  \newtheorem{theorem}[equation]{Theorem}
  \newtheorem*{theorem*}{Theorem}
  \newtheorem{proposition}[equation]{Proposition}

\theoremstyle{definition}
  \newtheorem{definition}[equation]{Definition}
  \newtheorem{lemma}[equation]{Lemma}
  \newtheorem{corollary}[equation]{Corollary}
  \newtheorem{notation}[equation]{Notation}

\theoremstyle{remark}
  
  \newtheorem{remark}[equation]{Remark}

 \DeclareFontFamily{U}{manual}{}
 \DeclareFontShape{U}{manual}{m}{n}{ <->  manfnt }{}
 \newcommand{\manfntsymbol}[1]{%
    {\fontencoding{U}\fontfamily{manual}\selectfont\symbol{#1}}}

\makeatletter
 \newenvironment{warning}[1][]{
    \begin{trivlist} \item[] \noindent%
      \begingroup\hangindent=2pc\hangafter=-2
      \clubpenalty=10000%
      \hbox to0pt{\hskip-\hangindent\manfntsymbol{127}\hfill}\ignorespaces%
      \refstepcounter{equation}\textbf{Warning~\theequation}%
      \@ifnotempty{#1}{\the\thm@notefont \ (#1)}\textbf{.}
      \let\p@@r=\par \def\p@r{\p@@r \hangindent=0pc} \let\par=\p@r}%
    {\hspace*{\fill}
     \endgraf\endgroup\end{trivlist}}
 \newenvironment{example}[1][]{
   \refstepcounter{equation}
   \begin{proof}[Example~\theequation%
   \@ifnotempty{#1}{ (#1)}.]
   }
  {\end{proof}}
\makeatother

 \DeclareFontFamily{OT1}{pzc}{}
 \DeclareFontShape{OT1}{pzc}{m}{it}{<-> s * [1.100] pzcmi7t}{}
 \DeclareMathAlphabet{\mathpzc}{OT1}{pzc}{m}{it}

\DeclareMathOperator{\ext}{Ext}
\DeclareMathOperator{\sat}{Sat}
\newcommand{\tor}{\mathrm{tor}}
\newcommand{\hhat}[1]{\widehat{#1}}
\renewcommand{\ng}{\text{-1}}
\newcommand{\qcoh}{\text{\sf QCoh}}

\renewcommand{\b}{\mathbf{b}}
\newcommand{\id}{\mathrm{id}}

\renewcommand{\AA}{\mathbb A}
\newcommand{\A}{\mathcal A}

\newcommand{\gp}{\mathrm{gp}}

\newcommand{\D}{\mathcal D}
\renewcommand{\L}{\mathcal L}

\newcommand{\X}{\mathcal X}
\newcommand{\Y}{\mathcal Y}
\newcommand{\Z}{\mathcal Z}
\newcommand{\ZZ}{\mathbb Z}
\newcommand{\PP}{\mathbb P}
\newcommand{\QQ}{\mathbb Q}
\newcommand{\NN}{\mathbb N}
\newcommand{\GG}{\mathbb G}
\DeclareMathOperator{\cok}{cok}
\newcommand{\bbar}[1]{\overline{#1}}
\newcommand{\ttilde}[1]{\widetilde{#1}}
\newcommand{\K}{\mathcal{K}}

\renewcommand{\setminus}{\smallsetminus}

\let\hom\relax
\DeclareMathOperator{\hom}{Hom}
\renewcommand{\O}{\mathcal{O}}

\DeclareMathOperator{\spec}{Spec}
\DeclareMathOperator{\Spec}{\mathcal{S}\!\mathpzc{pec}}
\DeclareMathOperator{\pic}{Pic}

\newcommand{\smat}[1]{\left(\begin{smallmatrix}#1\end{smallmatrix}\right)}

\newcommand{\F}{\mathcal F}
\renewcommand{\a}{\mathbf a} 

\DeclareMathOperator{\im}{im}

 \def\ari[#1]{\ar@{^(->}[#1]}
 \def\are[#1]{\ar[#1]^{\txt{\'et}}}
 \def\areh[#1]{\ar[#1]|{\txt{$H$-eq}}^{\txt{\'et}}}
 \def\ars[#1]{\ar@{->>}[#1]}
 \newcommand{\dplus}{\ar@{}[d]|{\mbox{$\oplus$}}}
 \newcommand{\dtimes}{\ar@{}[d]|{\mbox{$\times$}}}

\begin{document}
\title{Toric Stacks I: The Theory of Stacky Fans}
\author{Anton Geraschenko}
\author{Matthew Satriano}
\thanks{The second author is partially supported by NSF grant DMS-0943832.}
  \subjclass[2010]{
  14D23,
  14M25.
  }
\date{}

\begin{abstract}
 The purpose of this paper and its sequel \cite{toricartin2} is to introduce and develop a theory of toric stacks which encompasses and extends the notions of toric stacks defined in \cite{lafforgue,bcs,fmn,iwanari,can,tyomkin}, as well as classical toric varieties.

 In this paper, we define a \emph{toric stack} as the stack quotient of a toric variety by a subgroup of its torus (we also define a generically stacky version). Any toric stack arises from a combinatorial gadget called a \emph{stacky fan}. We develop a dictionary between the combinatorics of stacky fans and the geometry of toric stacks, stressing stacky phenomena such as canonical stacks and good moduli space morphisms.

 We also show that smooth toric stacks carry a moduli interpretation extending the usual moduli interpretations of $\PP^n$ and $[\AA^1/\GG_m]$. Indeed, smooth toric stacks precisely solve moduli problems  specified by (generalized) effective Cartier divisors with given linear relations and given intersection relations. Smooth toric stacks therefore form a natural closure to the class of moduli problems introduced for smooth toric varieties and smooth toric DM stacks in \cite{cox:smooth} and \cite{perroni}, respectively.

 We include a plethora of examples to illustrate the general theory. We hope that this theory of toric stacks can serve as a companion to an introduction to stacks, in much the same way that toric varieties can serve as a companion to an introduction to schemes.
\end{abstract}
\maketitle

\tableofcontents

\section{Introduction}\label{sec:intro}

A number of theories of toric stacks have recently been introduced \cite{lafforgue, bcs,fmn,iwanari,can,tyomkin}. There are several reasons why one may be interested in developing such a theory. First, these stacks provide a natural place to test conjectures and develop intuition about algebraic stacks, in much the same way that toric varieties do for schemes. After developing an adequate theory, they are easy to work with combinatorially, just as toric varieties are. Second, in some situations toric stacks can serve as better-behaved substitutes for toric varieties. A toric variety has a canonical overlying smooth stack.
It is sometimes easier to prove results on the smooth stack and ``push them down'' to the toric variety. Third, with the appropriate machinery, one can show that an ``abstract toric stack'' (e.g.~the closure of a torus within some stack of interest) often arises from a combinatorial stacky fan. The combinatorial theory of stacky fans then allows one to effectively investigate the stack in question.

We are aware of three kinds of toric stacks in the literature.
\begin{description}
 \item[Lafforgue's Toric Stacks.] In \cite{lafforgue}, Lafforgue defines a toric stack to be the stack quotient of a toric variety by its torus. These stacks are very ``small'' in the sense that they have a dense open point. They are rarely smooth. 
 \item[Smooth Toric Stacks.] Borisov, Chen, and Smith define smooth toric Deligne-Mumford stacks in \cite{bcs}. These are the stacks studied in \cite{fmn} and \cite{iwanari}. They are smooth and have simplicial toric varieties as their coarse moduli spaces. The second author generalized this approach in \cite{can} to include certain smooth toric Artin stacks which have toric varieties as their good moduli spaces.
 \item[Toric Varieties and Singular Toric Stacks.] Toric varieties are neither ``small'' nor smooth in general, so the standard theory of toric varieties is not subsumed by the above approaches. In \cite[\S 4]{tyomkin}, Tyomkin introduces a definition of toric stacks which includes all toric varieties (in particular, they may be singular). However, these toric stacks always have finite diagonal, and the definition is very local. 
\end{description}

We define a toric stack to be the stack quotient of a normal toric variety $X$ by a subgroup $G$ of its torus $T_0$. Note that the stack $[X/G]$ has a dense open torus $T=T_0/G$ which acts on $[X/G]$. Therefore, such toric stacks have trivial generic stabilizer. Many of the existing theories of toric stacks allow for generic stackiness. In order to encompass these in our theory, we enlarge our class of toric stacks to non-strict toric stacks with the following observation.

An integral $T$-invariant substack of $[X/G]$ is necessarily of the form $[Z/G]$ where $Z\subseteq X$ is an integral $T_0$-invariant subvariety of $X$.\footnote{Note that $Z$ must be irreducible because $G$ cannot permute the irreducible $T_0$-invariant subvarieties of $X$.} The subvariety $Z$ is naturally a toric variety whose torus $T'$ is a quotient of $T_0$. The quotient stack $[Z/G]$ contains a dense open ``stacky torus'' $[T'/G]$ which acts on $[Z/G]$.
\begin{definition}\label{def:toric-stack}
 In the notation of the above two paragraphs, a \emph{toric stack} is an Artin stack of the form $[X/G]$, together with the action of the torus $T=T_0/G$. A \emph{non-strict toric stack} is an Artin stack which is isomorphic to an integral closed torus-invariant substack of a toric stack, i.e.~is of the form $[Z/G]$, together with the action of the stacky torus $[T'/G]$.
\end{definition}

\noindent  This definition encompasses and extends the three kinds of toric stacks listed above:
 \begin{itemize}
  \item Taking $G$ to be trivial, we see that any toric variety $X$ is a toric stack.
  \item Smooth toric Deligne-Mumford stacks in the sense of \cite{bcs,fmn,iwanari} are smooth non-strict toric stacks which happen to be separated and Deligne-Mumford. See Remarks \ref{rmk:compatibility-with-other-toric-stacks} and \ref{rmk:non-strict-is-closed-substack}.
  \item Toric stacks in the sense of \cite{lafforgue} are toric stacks that have a dense open point (i.e.~toric stacks for which $G=T_0$).
  \item A toric Artin stack in the sense of \cite{can} is a smooth non-strict toric stack which has finite generic stabilizer and which has a toric variety of the same dimension as a good moduli space. See Sections \ref{sec:fantastacks} and \ref{sec:gms-morphisms}.
  \item Toric stacks in the sense of \cite{tyomkin} are toric stacks as well. This follows from the main theorem of \cite{toricartin2}, stated below. See \cite[Remark 6.2]{toricartin2} for more details.
 \end{itemize}

Just as toric varieties can be understood in terms of fans, toric stacks can be understood in terms of combinatorial objects called stacky fans. This paper develops a rich dictionary between the combinatorics of stacky fans and the geometry of toric stacks. In contrast, \cite{toricartin2} focuses on how to show a given stack is toric (and so amenable to the combinatorial anlaysis of stacky fans). A classical result (see for example \cite[Corollary 3.1.8]{cls}) shows that if $X$ is a finite type \emph{scheme} with a dense open torus $T$ whose action on itself extends to $X$, then $X$ is a toric variety if and only if it is normal and separated. Analogously, the main result of \cite{toricartin2} is:

\begin{theorem*}[{\cite[Theorem 6.1]{toricartin2}}]
 Let $\X$ be an Artin stack of finite type over an algebraically closed field $k$ of characteristic $0$. Suppose $\X$ has an action of a torus $T$ and a dense open substack which is $T$-equivariantly isomorphic to $T$. Then $\X$ is a toric stack if and only if the following conditions hold:
 \begin{enumerate}
  \item $\X$ is normal,
  \item $\X$ has affine diagonal,
  \item geometric points of $\X$ have linearly reductive stabilizers, and
  \item \label{intro:global-type} every point of $[\X/T]$ is in the image of an \'etale representable map from a stack of the form $[U/G]$, where $U$ is quasi-affine and $G$ is an affine group.\footnote{A forthcoming result of Alper, Hall, and Rydh shows that this condition is superfluous. See \cite[Remark 4.4(0)]{toricartin2}.}
 \end{enumerate}
\end{theorem*}

\subsection*{Organization of This Paper}

In Section \ref{sec:definitions}, we define \emph{(non-strict) stacky fans}.  Once the definitions are in place, it will be clear that any morphism of stacky fans induces a toric morphism of the corresponding toric stacks.  In Section \ref{sec:morphisms}, we prove that the converse is true as well: any toric morphism of toric stacks is induced by a morphism of stacky fans (Theorem \ref{thm:morphisms-come-from-fans}).

Sections \ref{sec:definitions} and \ref{sec:morphisms} provide a sufficient base to generate interesting examples. In Section \ref{sec:fantastacks}, we highlight a particularly easy to handle class of toric stacks, which we call \emph{fantastacks}. Though non-strict toric stacks are considerably more general than fantastacks, it is sometimes easiest to understand a non-strict toric stack in terms of its relation to some fantastack. The stacks defined in \cite{bcs} and \cite{can} which have no generic stabilizer are fantastacks; those with non-trivial generic stabilizer are closed substacks of fantastacks.

Section \ref{sec:canonical-stacks} is devoted to the construction of the \emph{canonical stack} over a toric stack. The main result is Proposition \ref{prop:canonical-stack}, which justifies the terminology by showing that canonical stacks have a universal property. Canonical stacks are minimal ``stacky resolutions'' of singularities. Heuristically, the existence of such a resolution is desirable because it is sometimes possible to prove theorems on the smooth resolution and then descend them to the singular base. Indeed, the main theorem of Toric Stacks II, \cite[Theorem 6.1]{toricartin2}, is proved in this way.

In Section \ref{sec:gms-morphisms}, we prove Theorem \ref{thm:main-gms}, a combinatorial characterization of toric morphisms which are \emph{good moduli space morphisms} in the sense of \cite{Alper:good}. Good moduli space morphisms generalize the notion of a coarse moduli space and that of a good quotient in the sense of \cite{git}. Good moduli space morphisms are of central interest in the theory of moduli, so it is useful to have tools for easily identifying and handling many examples.

In Section \ref{sec:moduli}, we prove a \emph{moduli interpretation} for smooth toric stacks (Theorem \ref{thm:moduli}). That is, we characterize morphisms to smooth toric stacks from an arbitrary source, rather than only toric morphisms from toric stacks. The familiar moduli interpretation of $\PP^n$ is that specifying a morphism to $\PP^n$ is equivalent to specifying a line bundle, together with $n+1$ sections that generate it. Cox generalized this interpretation to smooth toric varieties in \cite{cox:smooth}, and Perroni further generalized it to smooth toric Deligne-Mumford stacks in \cite{perroni}. Smooth toric stacks are the natural closure of this class of moduli problems. In other words, any moduli problem of the same sort as described by Cox and Perroni is represented by a smooth toric stack (see Remark \ref{rmk:moduli-natural-class}).

\begin{remark}
 Just as toric varieties can be defined over an arbitrary base, much of this paper can be done over an arbitrary base, but we work over an algebraically closed field in order to avoid imposing confusing hypotheses (e.g.~every subgroup of a torus we consider is required to be diagonalizable). A stacky fan defines a toric stack over an arbitrary base, and morphisms of stacky fans induce morphisms of toric stacks. However, if the base is disconnected, not every toric morphism is induced by a morphism of stacky fans.
\end{remark}

\begin{remark}[The Log Geometric Approach]
 There is yet another approach to toric geometry, namely that of log geometry. In this paper, we do not develop this approach to toric stacks. We refer the interested reader to \cite[\S\S 5--6]{can}, in which the log geometric approach is taken for fantastacks.
\end{remark}

\subsection*{Logical Dependence of Sections}

The logical dependence of sections is roughly as follows:
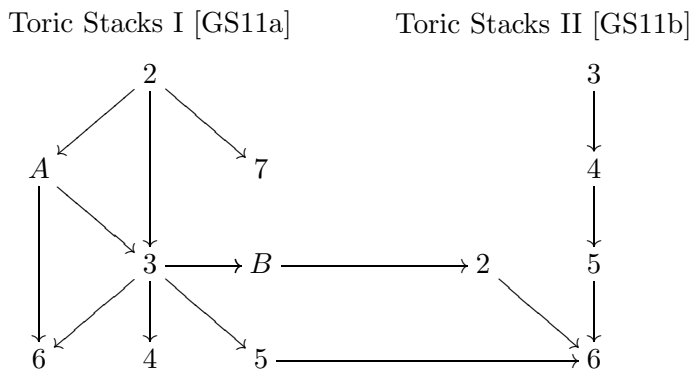
\[\xymatrix@!0 @C+1.5pc @R+1pc{
   & {\ref{sec:definitions}}\ar[dd]\ar[dl]\ar[dr] \POS p+(0,.7) *\txt{Toric Stacks I \cite{toricartin1}} & & & & {3}\ar[d] \POS p+(-.7,.7) *\txt{Toric Stacks II \cite{toricartin2}}\\
 {\ref{sec:ses-of-G_betas}}\ar[dd]\ar[dr] & & {\ref{sec:moduli}} & & & {4}\ar[d]\\
 & {\ref{sec:morphisms}}\ar[dl]\ar[d]\ar[dr]\ar[r]\ar[r] & {\ref{sec:non-uniqueness-of-fans}}\ar[rr] & & {2} \ar[dr] & {5} \ar[d]
 \\
 {\ref{sec:gms-morphisms}} & {\ref{sec:fantastacks}} &  {\ref{sec:canonical-stacks}}\ar[rrr] & & & {6}
}\]

\subsection*{Acknowledgments}
We thank Jesse Kass and Martin Olsson for conversations which helped get this project started, and Vera Serganova and the MathOverflow community (especially Torsten Ekedahl, Jim Humphreys, Peter McNamara, David Speyer, and Angelo Vistoli) for their help with several technical points. We also thank Smiley for helping to track down many references. Finally, we would like to thank the anonymous referee for helpful suggestions and interesting questions.

\section{Definitions}\label{sec:definitions}

For a brief introduction to algebraic stacks, we refer the reader to \cite[Chapter 1]{olsson:abelian}. For a more detailed treatment, we refer to \cite{vistoli:fibered} or \cite{lmb}. If the reader is unfamiliar with stacks, we encourage her to continue reading, simply treating $[X/G]$ as a formal quotient of a scheme $X$ by an action of a group $G$. Just as it is possible to learn the theory of toric varieties as a means of learning about varieties in general, we hope the theory of toric stacks can serve as an introduction to the theory of algebraic stacks.

We refer the reader to \cite[Chapter 1]{fulton} or \cite[Chapter 3]{cls} for the standard correspondence between fans on lattices and toric varieties, and for basic results about toric varieties. We will follow the notation in \cite{cls} whenever possible.

\begin{definition}\label{def:T_L}
 If $L$ is a lattice (i.e.~a finitely generated free abelian group), we denote by $T_L$ the torus $D(L^*)=\hom_\gp(\hom_\gp(L,\ZZ),\GG_m)$ whose lattice of 1-parameter subgroups is naturally isomorphic to $L$.
\end{definition}
\begin{remark}\label{rmk:Cartier-dual}
 Here, the functor $D(-)$ is the \emph{Cartier dual} $\hom_\gp(-,\GG_m)$. It is an anti-equivalence of categories between finitely generated abelian groups and diagonalizable group schemes. See \cite[Expos\'e VIII]{sga3}.

 The functor $(-)^*$ is the usual dual for finitely generated abelian groups, $\hom_\gp(-,\ZZ)$.
\end{remark}

\begin{remark}[$\cok\beta$ finite $\Longleftrightarrow \beta^*$ injective]\label{rmk:cok-finite-iff-dual-injective}
  We will often use the fact that a morphism $\beta\colon L\to N$ of \emph{arbitrary} finitely generated abelian groups has finite cokernel if and only if $\beta^*\colon N^*\to L^*$ is injective. This is immediate from the fact that $\hom(-,\ZZ)$ is left exact, that $L\to N\to \cok \beta\to 0$ is exact, and that $\hom(\cok\beta,\ZZ)=0$ if and only if $\cok\beta$ is finite. Note also that this property depends only on the quasi-isomorphism class of the mapping cone $C(\beta) = [L\to N]$.
\end{remark}

\subsection{The Toric Stack of a Stacky Fan}
\label{subsec:toric-stack-of-stacky-fan}

\begin{definition}\label{def:stacky-fan}
 A \emph{stacky fan} is a pair $(\Sigma,\beta)$, where $\Sigma$ is a fan on a lattice $L$ and $\beta\colon L\to N$ is a homomorphism to a lattice $N$ so that $\cok\beta$ is finite.
\end{definition}

A stacky fan $(\Sigma,\beta)$ gives rise to a toric stack as follows. Let $X_\Sigma$ be the toric variety associated to $\Sigma$ (see \cite[\S 1.4]{fulton} or \cite[\S 3.1]{cls}). The map $\beta^*\colon N^*\to L^*$ induces a homomorphism of tori $T_\beta\colon T_L\to T_N$, naturally identifying $\beta$ with the induced map on lattices of 1-parameter subgroups. Since $\cok\beta$ is finite, $\beta^*$ is injective, so $T_\beta$ is surjective. Let $G_\beta=\ker(T_\beta)$. Note that $T_L$ is the torus of $X_\Sigma$, and $G_\beta\subseteq T_L$ is a subgroup.

\begin{definition}\label{def:fan->toric-stack}
 Using the notation in the above paragraph, if $(\Sigma,\beta)$ is a stacky fan, we define the toric stack $\X_{\Sigma,\beta}$ to be $[X_\Sigma/G_\beta]$, with the torus $T_N=T_L/G_\beta$.
\end{definition}

Conversely, every toric stack arises from a stacky fan, since every toric stack is of the form $[X/G]$, where $X$ is a toric variety and $G\subseteq T_0$ is a subgroup of its torus. Associated to $X$ is a fan $\Sigma$ on the lattice $L=\hom_\gp(\GG_m,T_0)$ of 1-parameter subgroups of $T_0$. The surjection of tori $T_0\to T_0/G$ induces a homomorphism of lattices of 1-parameter subgroups, $\beta\colon L\to N:=\hom_\gp(\GG_m,T_0/G)$. The dual homomorphism $\beta^*\colon N^*\to L^*$ is the induced homomorphism of characters. Since $T_0\to T_0/G$ is surjective, $\beta^*$ is injective, so $\cok\beta$ is finite. Thus $(\Sigma,\beta)$ is a stacky fan. It is straightforward to check that $[X/G]=\X_{\Sigma,\beta}$.

\begin{center}
\begin{tabular}{@{\extracolsep{1ex}}c|cccccc}
 Example & \ref{eg:toric-varieties} & \ref{eg:A_1} & \ref{eg:A^2/G_m} & \ref{eg:non-canonical-presentations} & \ref{eg:double-origin} & \ref{eg:A^1/mu_2}\\ \hline
 \xymatrix{L\ar[d]^\beta \POS p+(0,1.2) *+{\Sigma} \\ N} &
 \xymatrix{N\ar[d]^\id \POS p+(0,1.2) *+{\Sigma} \\ N} &
 \xymatrix{\ZZ^2\ar[d]^{\smat{1&0\\ 1&2}} \POS p+(0,1.2) *+{
   \begin{tikzpicture}
      \clip (-.1,-.1) rectangle (1,1);
     \filldraw[fill=lightgray]
      (0,0) -- (9,0) -- (9,9) -- (0,9) -- cycle;
    \draw[thick,<->] (0,1) -- (0,0) -- (1,0);
   \end{tikzpicture}} \\ \ZZ^2} &
 \xymatrix{\ZZ^2\ar[d]^{\smat{1&0}} \POS p+(0,1.2) *+{
   \begin{tikzpicture}
      \clip (-.1,-.1) rectangle (1,1);
     \filldraw[fill=lightgray]
      (0,0) -- (9,0) -- (9,9) -- (0,9) -- cycle;
    \draw[thick,<->] (0,1) -- (0,0) -- (1,0);
   \end{tikzpicture}} \\ \ZZ} &
 \xymatrix{\ZZ^2\ar[d]^{\smat{1&\ng}} \POS p+(0,1.2) *+{
   \tikz
    \draw[thick,<->] (0,1) -- (0,0) -- (1,0);
   } \\ \ZZ} &
 \xymatrix{\ZZ^2\ar[d]^{\smat{1&1}} \POS p+(0,1.2) *+{
   \tikz
    \draw[thick,<->] (0,1) -- (0,0) -- (1,0);
   } \\ \ZZ} &
 \xymatrix{\ZZ\ar[d]^{2} \POS p+(0,0.7) *+{
   \tikz
    \draw[thick,->] (0,0) -- (1,0);
   } \\ \ZZ}
\end{tabular}
\end{center}

\begin{example}[Toric Varieties]\label{eg:toric-varieties}
 Suppose $\Sigma$ is a fan on a lattice $N$. Letting $L=N$ and $\beta=\id_N$, we see that the induced map $T_N\to T_N$ is the identity map, so $G_\beta$ is trivial. So $\X_{\Sigma,\beta}$ is the toric variety $X_\Sigma$.
\end{example}

\begin{example}\label{eg:A_1}
 Here $X_\Sigma=\AA^2$. We have that $\beta^*$ is given by $\smat{1&1\\ 0&2}\colon \ZZ^2\to \ZZ^2$, so the induced map on tori $\GG_m^2\to \GG_m^2$ is given by $(s,t)\mapsto (st,t^2)$. The kernel is $G_\beta = \mu_2=\{(\zeta,\zeta)|\zeta^2=1\}\subseteq \GG_m^2$.

 \begin{window}[0,l,%
   
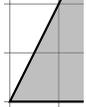
{\begin{tikzpicture}[scale=.65]
     \clip (-.1,-.1) rectangle (1.5,2.1);
     \filldraw[fill=lightgray]
      (0,0) -- (9,0) -- (4,8) -- cycle;
     \draw[help lines] (-1,-1) grid (3,3);
     \draw[thick] (4,8) -- (0,0) -- (8,0);
 \end{tikzpicture}%
 },]
 So we see that $\X_{\Sigma,\beta}=[\AA^2/\mu_2]$, where the action of $\mu_2$ is given by $\zeta\cdot (x,y)=(\zeta x,\zeta y)$. Note that this is a smooth stack. It is distinct from the singular toric variety with the fan shown to the left.\qedhere
 \end{window}
\end{example}

Since quotients of subschemes of $\AA^n$ by subgroups of $\GG_m^n$ appear frequently, we often include the weights of the action in the notation.
\begin{notation}\label{not:weights}
 Let $G\hookrightarrow \GG_m^n$ be the subgroup corresponding to the surjection $\ZZ^n\to D(G)$. Let $g_i$ be the image of $e_i$ in $D(G)$. Let $X\subseteq \AA^n$ be a $\GG_m^n$-invariant subscheme. We denote the quotient $[X/G]$ by $[X/_{\smat{g_1&\cdots&g_n}}G]$.
\end{notation}
In this notation, the stack in Example \ref{eg:A_1} would be denoted $[\AA^2/_{\smat{1&1}}\mu_2]$.

\begin{example}\label{eg:A^2/G_m}
 Again we have that $X_\Sigma=\AA^2$. This time $\beta^*=\smat{1\\0}\colon \ZZ\to \ZZ^2$, which induces the homomorphism $\GG_m^2\to \GG_m$ given by $(s,t)\mapsto s$. Therefore, $G_\beta = \GG_m = \{(1,t)\}\subseteq \GG_m^2$, so $\X_{\Sigma,\beta} = [\AA^2/_{\smat{0&1}}\GG_m]\cong \AA^1\times [\AA^1/\GG_m]$.
\end{example}

\begin{example}\label{eg:non-canonical-presentations}
 This time $X_\Sigma=\AA^2\setminus \{(0,0)\}$. We see that $\beta^*=\smat{1\\ \ng}\colon \ZZ\to \ZZ^2$, which induces the morphism $\GG_m^2\to \GG_m$ given by $(s,t)\mapsto st^{-1}$. So $G_\beta = \GG_m=\{(t,t)\}\subseteq \GG_m^2$. We then have that $\X_{\Sigma,\beta}=[(\AA^2\setminus \{(0,0)\})/_{\smat{1&1}}\GG_m]=\PP^1$.
\end{example}

\begin{warning}\label{warn:non-canonical-presentations}
 Examples \ref{eg:toric-varieties} and \ref{eg:non-canonical-presentations} show that non-isomorphic stacky fans (see Definition \ref{def:morphism-of-stacky-fans}) can give rise to isomorphic toric stacks. The two presentations $[(\AA^2\setminus\{(0,0)\})/_{\smat{1&1}}\GG_m]$ and $[\PP^1/\{e\}]$ of the same toric stack are produced by different stacky fans. In Theorem \ref{thm:isomorphisms}, we determine when different stacky fans give rise to the same toric stack.
\end{warning}

\begin{example}[The non-separated line]\label{eg:double-origin}
 Again we have that $X_\Sigma=\AA^2\setminus \{(0,0)\}$. However, this time we see that $\beta^*=\smat{1\\ 1}\colon \ZZ\to \ZZ^2$, which induces the homomorphism $\GG_m^2\to \GG_m$ given by $(s,t)\mapsto st$. Therefore, $G_\beta = \GG_m=\{(t,t^{-1})\}\subseteq \GG_m^2$. So we have that $\X_{\Sigma,\beta}=[(\AA^2\setminus \{(0,0)\})/_{\smat{1&\ng}}\GG_m]$ is the affine line $\AA^1$ with a doubled origin.

 This example shows that there are toric stacks which are schemes, but are not toric varieties because they are non-separated.
\end{example}

\begin{example}\label{eg:A^1/mu_2}
 Here $X_\Sigma=\AA^1$, and $\beta^*=2\colon \ZZ\to \ZZ$, which induces the map $\GG_m\to \GG_m$ given by $t\mapsto t^2$. So $G_\beta = \mu_2\subseteq \GG_m$ and $\X_{\Sigma,\beta} = [\AA^1/\mu_2]$.
\end{example}

\subsection{Non-strict Stacky Fans and Non-strict Toric Stacks}

In this subsection, we generalize Definitions \ref{def:stacky-fan} and \ref{def:fan->toric-stack} to allow for stacks with non-trivial generic stabilizer.

\begin{definition}\label{def:G_beta}
 Suppose $f\colon A\to B$ is a homomorphism of finitely generated abelian groups so that $\ker f$ is free. For $i=0,1$, let $D(G_f^i)$ be $H^i(C(f)^*)$, where $C(f)=[A\xrightarrow{f}B]$ is the mapping cone of $f$ and $(-)^*$ is the derived functor $R\hom_\gp(-,\ZZ)$. We define $G_f^i$ to be the diagonalizable groups corresponding to $D(G_f^{i})$, and we define $G_f = G_f^0\oplus G_f^1$.
\end{definition}
Note that the homomorphism $A^*\to D(G_f^{1})$ induces a homomorphism $G_f\to D(A^{*})$ (that is trivial on $G_f^0$). In the case where $A$ and $B$ are lattices, $H^i(C(f)^*)$ are simply the kernel ($i=0$) and cokernel ($i=1$) of $f^*$. If we additionally assume $f$ has finite cokernel (as was the case in Section \ref{subsec:toric-stack-of-stacky-fan}), then $f^*$ has no kernel, so $G^0_f$ is trivial. In particular, the notation is consistent with the notation in the paragraph above Definition \ref{def:fan->toric-stack}, with $f=\beta$.

We now generalize Definitions \ref{def:stacky-fan} and \ref{def:fan->toric-stack}.

\begin{definition}\label{def:non-strict-fan}
 A \emph{non-strict stacky fan} is a pair $(\Sigma,\beta)$, where $\Sigma$ is a fan on a lattice $L$, and $\beta\colon L\to N$ is a homomorphism to a finitely generated abelian group.
\end{definition}

\begin{definition}\label{def:fan->non-strict}
 If $(\Sigma,\beta)$ is a non-strict stacky fan, we define $\X_{\Sigma,\beta}$ to be $[X_\Sigma/G_\beta]$, where the action of $G_\beta$ on $X_\Sigma$ is induced by the homomorphism $G_\beta\to D(L^*)=T_L$.
\end{definition}
\begin{remark}\label{rmk:compatibility-with-other-toric-stacks}
 If $\beta$ is assumed to have finite cokernel, this construction of $\X_{\Sigma,\beta}$ essentially agrees with the ones in \cite[\S 3]{bcs} and \cite[\S 5]{can}. However, those constructions effectively impose additional conditions on $\Sigma$ (e.g.~that $\Sigma$ is a subfan of the fan of $\AA^n$) since it is required to be induced by a fan on $N\otimes_\ZZ \QQ$.
\end{remark}

\begin{remark}\label{rmk:non-strict-is-closed-substack}
 We now give a more explicit description of $\X_{\Sigma,\beta}$, which also has the benefit of demonstrating that it is a non-strict toric stack according to Definition \ref{def:toric-stack}. See Example \ref{eg:non-strict-as-closed-substack} for an illustration of this approach.

 Let $(\Sigma,\beta\colon L\to N)$ be a non-strict stacky fan. Let
 \[
 \ZZ^s\xrightarrow Q \ZZ^r\to N\to 0
 \]
 be a presentation of $N$, and let $B\colon L\to \ZZ^r$ be a lift of $\beta$.

 Define the fan $\Sigma'$ on $L\oplus \ZZ^s$ as follows. Let $\tau$ be the cone generated by $e_1,\dots, e_s\in \ZZ^s$. For each $\sigma\in \Sigma$, let $\sigma'$ be the cone spanned by $\sigma$ and $\tau$ in $L\oplus \ZZ^s$. Let $\Sigma'$ be the fan generated by all the $\sigma'$. Corresponding to the cone $\tau$, we have the closed subvariety $Y\subseteq X_{\Sigma'}$, which isomorphic to $X_\Sigma$ since $\Sigma$ is the \emph{star} (sometimes called the \emph{link}) of $\tau$ \cite[Proposition 3.2.7]{cls}. We define

 \[\xymatrix@R-2pc @C-1pc{
 \llap{$\beta'=B\oplus Q\colon\,$}L\oplus \ZZ^s\ar[r] & \ZZ^r\\
 (l,\a)\ar@{|->}[r] & B(l)+Q(\a).
 }\]

 Then $(\Sigma',\beta')$ is a stacky fan and we see that $\X_{\Sigma,\beta}\cong [Y/G_{\beta'}]$. Note that $C(\beta')$ is quasi-isomorphic to $C(\beta)$, so $G_{\beta'}\cong G_\beta$.
\end{remark}

\begin{remark}\label{rmk:smooth-gen-stacky<smooth}
 Note that if $\sigma$ is a smooth cone,\footnote{A \emph{smooth cone} is a cone whose corresponding toric variety is smooth. See \cite[Definition 1.2.16]{cls}.} then the cone spanned by $\sigma$ and $\tau$ is also a smooth cone. So if $\X_{\Sigma,\beta}$ is a \emph{smooth} non-strict toric stack, then it is a closed substack of a \emph{smooth} toric stack.
\end{remark}


The following definition will often be useful when dealing with stacky fans.
\begin{definition}\label{def:saturated}
 Suppose $B$ is a finitely generated abelian group and $A\subseteq B$ is a subgroup. The \emph{saturation of $A$ in $B$} is the subgroup
 \[
  \sat_BA = \{b\in B|n\cdot b\in A\text{ for some }n\in \ZZ_{>0}\}.
 \]
 We say $A$ is \emph{saturated in $B$} if $A=\sat_BA$. We say that a homomorphism $f\colon A\to B$ is \emph{saturated} if $f(A)$ is saturated in $B$.
\end{definition}

\begin{remark}
 Saturated morphisms are precisely morphisms whose cokernels are lattices. In particular, the image of a saturated morphism has a direct complement.
\end{remark}

\begin{remark}[On the condition ``$\cok\beta$ is finite'']\label{rmk:beta-finite-cokernel}
 Since the action of $G_\beta^0$ on $X_\Sigma$ is trivial, we have that $\X_{\Sigma,\beta}=[X_\Sigma/G_\beta^1]\times BG_\beta^0$. It is often easiest to treat the factor of $BG_\beta^0$ separately. Let $N_1=\sat_N(\beta(L))$, let $N_0$ be a direct complement to $N_1$, and let $\beta_1\colon L\to N_1$ be the factorization of $\beta$ through $N_1$. Then $G_\beta^0 = D(N_0^*) = \GG_m^{\mathrm{rk}(N_0)}$ and $[X_\Sigma/G_\beta^1] = \X_{\Sigma,\beta_1}$. We therefore typically assume $\beta$ has finite cokernel (equivalently, that $N_0=0$ or that $\X_{\Sigma,\beta}$ has finite generic stabilizer), with the understanding that the general case can usually be handled by replacing $\beta$ by $\beta_1$.
\end{remark}

\begin{remark}[On the non-strict case]\label{rmk:non-strict-case}
 In this paper, we work primarily with toric stacks, since non-strict toric stacks can be described as closed substacks. One reason for this focus is that we would like to avoid discussing actions of stacky tori in general. We refer the interested reader to \cite[\S 1.7, \S 2, and Appendix B]{fmn} for a discussion on stacky tori and their actions.
\end{remark}

\section{Morphisms of Toric Stacks}\label{sec:morphisms}

The main goal of this section is to define morphisms of toric stacks and stacky fans, and to show (in Theorem \ref{thm:morphisms-come-from-fans}) that every morphism of toric stacks is induced by a morphism of stacky fans.

\begin{definition}\label{def:toric-morphism}
 A \emph{toric morphism} or a \emph{morphism of {\rm(}non-strict{\rm)} toric stacks} is a morphism which restricts to a homomorphism of (stacky) tori and is equivariant with respect to that homomorphism.
\end{definition}

\begin{definition}\label{def:morphism-of-stacky-fans}
 A \emph{morphism of non-strict stacky fans} $(\Sigma,\beta\colon L\to N)\to (\Sigma',\beta'\colon L'\to N')$ is a pair of group morphisms $\Phi\colon L\to L'$ and $\phi\colon N\to N'$ so that $\beta'\circ \Phi=\phi\circ \beta$ and so that for every cone $\sigma\in \Sigma$, $\Phi(\sigma)$ is contained in a cone of $\Sigma'$.
\end{definition}
We typically draw a morphism of non-strict stacky fans as a commutative diagram:
\[\xymatrix{
 L\ar[d]_\beta \ar[r]^\Phi \POS p+(0,.7) *+{\Sigma}="s" & L'\ar[d]^{\beta'} \POS p+(0,.7) *+{\Sigma'} \ar@{<-} "s"\\
 N\ar[r]^\phi & N'
}\]

A morphism of non-strict stacky fans $(\Phi,\phi)\colon (\Sigma,\beta)\to (\Sigma',\beta')$ induces a morphism of toric varieties $X_\Sigma\to X_{\Sigma'}$ and a compatible morphism of groups $G_\beta\to G_{\beta'}$, so it induces a toric morphism of non-strict toric stacks $\X_{(\Phi,\phi)}\colon \X_{\Sigma,\beta}\to \X_{\Sigma',\beta'}$.

\subsection{Toric Morphisms are Induced by Morphisms of Stacky Fans}
\begin{lemma}\label{lem:torsor-connected-component}
 Let $X$ be a connected scheme, $G$ a group scheme over $k$, and $P\to X$ a $G$-torsor. Suppose $Q\subseteq P$ is a connected component of $P$. Then $Q\to X$ is an $H$-torsor, where $H$ is the subgroup of $G$ which sends $Q$ to itself.
\end{lemma}
\begin{proof}
 Let $\phi$ be an automorphism of $P$. Since $Q$ is a connected component of $P$, $\phi(Q)$ is either equal to $Q$ or is disjoint from $Q$. It follows that $(G\times Q)\times_P Q=(H\times Q)\times_P Q\cong H\times Q$, where the map $G\times Q\to P$ is induced by the action of $G$ on $P$.

 The diagonal $Q\to Q\times_X Q\subseteq P\times_X Q$ is a section of the $G$-torsor $P\times_X Q\to Q$, so it induces a $G$-equivariant isomorphism $P\times_X Q\cong G\times Q$. We then have the following cartesian diagram:
 \[\xymatrix{
  \llap{$H\times Q\cong\;$}(G\times Q)\times_P Q\ar[r]\ar[d] & Q\times_X Q\ar[d]\ar[r] & Q\ar[d]\\
  G\times Q \ar@{}[r]|{\mbox{$\cong$}} & P\times_X Q\ar[r]\ar[d] & P\ar[d]\\
  & Q\ar[r] & X
 }\]
 In particular, the map $H\times Q\to Q\times_X Q$, given by $(h,q)\mapsto (h\cdot q,q)$, is an isomorphism. This shows that $Q$ is an $H$-torsor.
\end{proof}

\begin{theorem}\label{thm:morphisms-come-from-fans}
 Let $(\Sigma,\beta\colon L\to N)$ and $(\Sigma',\beta'\colon L'\to N')$ be stacky fans, and suppose $f\colon \X_{\Sigma,\beta}\to \X_{\Sigma',\beta'}$ is a toric morphism. Then there exists a stacky fan $(\Sigma_0,\beta_0)$ and morphisms $(\Phi,\phi)\colon (\Sigma_0,\beta_0)\to (\Sigma,\beta)$ and $(\Phi',\phi')\colon (\Sigma_0,\beta_0)\to (\Sigma',\beta')$ such that the following triangle commutes and $\X_{(\Phi,\phi)}$ is an isomorphism:
 \[\xymatrix@R-1.7pc{
  \X_{\Sigma_0,\beta_0}\ar[dd]^\wr_{\X_{(\Phi,\phi)}}\ar[dr]^{\X_{(\Phi',\phi')}}\\
  & \X_{\Sigma',\beta'}\\
  \X_{\Sigma,\beta}\ar[ur]_f
 }\]
\end{theorem}

\begin{remark}
  If $\Sigma$ is the fan of faces of a single cone, we may take $(\Sigma_0,\beta_0)=(\Sigma,\beta)$ by Corollary \ref{cor:morphisms-from-pointed}. If a toric stack $\X$ has no torus factor, there is a canonical (initial) choice of stacky fan $(\Sigma,\beta)$ so that $\X\cong \X_{\Sigma,\beta}$ and any toric morphism from $\X$ is induced by a morphism from the fan $(\Sigma,\beta)$ (see \cite[Remark 2.14]{toricartin2}).
\end{remark}

\begin{proof}[Proof of Theorem \ref{thm:morphisms-come-from-fans}]
 By assumption, $f$ restricts to a homomorphism of tori $T_N \to T_{N'}$, which induces a homomorphism of lattices of 1-parameter subgroups $\phi\colon N\to N'$.

 We define $Y=X_\Sigma\times_{\X_{\Sigma',\beta'}}X_{\Sigma'}$. Since $X_\Sigma\to \X_{\Sigma',\beta'}$ and $X_{\Sigma'}\to \X_{\Sigma',\beta'}$ are toric, $T_L\times_{T_{N'}}T_{L'}$ is a diagonalizable group. The connected component of the identity, $T_0\subseteq T_L\times_{T_{N'}}T_{L'}$, is a connected diagonalizable group, so it is a torus. Let $Y_0$ be the connected component of $Y$ which contains $T_0$, and let $G_\Phi$ be the kernel of the homomorphism $T_0\to T_L$. We then have the following diagram:
 \[\xymatrix@!0 @C+1pc{
  G_\Phi\ari[dd]\ari[rr] && G_{\beta'}\ari[dd]\ar@{=}[rr] && G_{\beta'}\ari[dd]\\
  \\
  T_0\ari[rr]\ari[dr] && T_L\times_{T_{N'}}T_{L'} \ar@{->>}[dd]|{\hole}\ar[rr] \ari[dr] && T_{L'}\ari[dr]\ar@{->>}[dd]|{\hole}\\
  & Y_0\ari[rr] && Y\ar[dd]\ar[rr] && X_{\Sigma'}\ar[dd]^{G_{\beta'}\text{-torsor}}\\
  && T_L\ari[dr]\ar[rr]|{\hole} && T_{N'}\ari[dr]\\
  & && X_\Sigma\ar[rr] && \X_{\Sigma',\beta'}
 }\]
 Since $X_\Sigma$ is normal and separated, and $Y$ is a $G_{\beta'}$-torsor over $X_\Sigma$, we have that $Y$ is normal and separated, so $Y_0$ is normal, separated, and connected. In particular, $Y_0$ is irreducible. We have that $T_0$ is an open subscheme of $Y_0$, and $T_0$ acts on $Y_0$ in a way that extends the multiplication, so $Y_0$ is a toric variety with torus $T_0$. Say it corresponds to a fan $\Sigma_0$ on the lattice $L_0$ of 1-parameter subgroups of $T_0$.

 Now $Y_0\to X_\Sigma$ and $Y_0\to X_{\Sigma'}$ are morphisms of toric varieties, so they are induced by morphisms of fans $\Sigma_0\to \Sigma$ and $\Sigma_0\to \Sigma'$. Defining $\beta_0$ to be the composition $L_0\xrightarrow\Phi L\xrightarrow\beta N$, we have morphisms of stacky fans
 \[\xymatrix{
  L \ar[d]_\beta \POS p+(0,.7) *+{\Sigma}="s" & L_0 \ar[l]_{\Phi}\ar[r]\ar[d]_{\beta_0} \POS p+(0,.7) *+{\Sigma_0}="s0" \ar "s" & L'\ar[d]^{\beta'} \POS p+(0,.7) *+{\Sigma'} \ar@{<-} "s0"\\
  N\ar@{=}[r] & N\ar[r]^\phi & N'
 }\]
 Note that $G_\Phi$ is the kernel of the surjection $T_0\to T_L$, so the notation is consistent with Definition \ref{def:G_beta}. By construction, $G_\Phi$ is the subgroup of $G_{\beta'}$ which takes $T_0$ to itself, so it is the subgroup which takes $Y_0=\overline{T_0}$ to itself. By Lemma \ref{lem:torsor-connected-component}, $Y_0$ is a $G_\Phi$-torsor over $X_\Sigma$.

 Since $T_0$ is a connected component of a group that surjects onto $T_L$, the induced morphism $T_0\to T_L$ is surjective, so $\Phi$ has finite cokernel. By Lemma \ref{lem:ses=>ses-of-G_betas}, $G_{\beta_0}$ is an extension of $G_\beta$ by $G_\Phi$. The morphism of stacky fans $(\Sigma_0,\beta_0)\to (\Sigma,\beta)$ induces the isomorphism $\X_{\Sigma_0,\beta_0}=[Y/G_{\beta_0}]\to [(Y/G_\Phi)/G_\beta]= \X_{\Sigma,\beta}$.

 On the other hand, the morphism $(\Sigma_0,\beta_0)\to (\Sigma',\beta')$ induces the morphism $[Y_0/G_{\beta_0}]\cong [X_\Sigma/G_\beta]\to \X_{\Sigma',\beta'}$.
\end{proof}

We conclude this section with the following general purpose proposition for studying product morphisms.
\begin{proposition}\label{prop:product-of-property-P}
 Let $\mathbf{P}$ be a property of morphisms which is stable under composition and base change. For $i=0,1$, let $(\Phi_i,\phi_i)\colon (\Sigma_i,\beta_i\colon L_i\to N_i)\to (\Sigma_i,\beta_i\colon L_i\to N_i)$ be a morphism of non-strict stacky fans. Then the product morphism $(\Phi_0\times \Phi_1,\phi_0\times \phi_1)$ induces a morphism of non-strict toric stacks which has property $\mathbf{P}$ if and only if each $(\Phi_i,\phi_i)$ does:
 \[\xymatrix@C+1pc{
  L_0\times L_1\ar[d]_{\beta_0\times \beta_1} \ar[r]^{\Phi_0\times \Phi_1} \POS p+(0,.7) *+{\Sigma_0\times \Sigma_1}="s" & L_0'\times L_1'\ar[d]^{\beta_0'\times \beta_1'} \POS p+(0,.7) *+{\Sigma_0'\times \Sigma_1'} \ar@{<-} "s"\\
  N_0\times N_1\ar[r]^{\phi_0\times \phi_1} & N_0'\times N_1'
 }\]
 (see \cite[Proposition 3.1.14]{cls} for basic facts about product fans).
\end{proposition}
\begin{proof}
 Products of morphisms with property $\mathbf P$ have property $\mathbf P$ because $f\times g = (f\times \id)\circ (\id\times g)$, and $f\times \id$ (resp.~$\id\times g$) is a base change of $f$ (resp.~$g$).

 The constructions of $\X_{\Sigma,\beta}$ from $(\Sigma,\beta)$ and of $\X_{(\Phi,\phi)}$ from $(\Phi,\phi)$ commute with products, so $(\Phi_0\times \Phi_1,\phi_0\times \phi_1)$ induces the product morphism $\X_{(\Phi_0\times \Phi_1,\phi_0\times \phi_1)}=\X_{(\Phi_0,\phi_0)}\times \X_{(\Phi_1,\phi_1)}$. It follows that if $\X_{(\Phi_i,\phi_i)}$ have property $\mathbf P$, then so does $\X_{(\Phi_0\times \Phi_1,\phi_0\times \phi_1)}$.

 Conversely, suppose $\X_{(\Phi_0\times \Phi_1,\phi_0\times \phi_1)}=\X_{(\Phi_0,\phi_0)}\times \X_{(\Phi_1,\phi_1)}$ has property $\mathbf P$. We have that $\X_{\Sigma_0',\beta_0'}$ has a $k$-point, the identity element of its torus, which induces a morphism $\X_{\Sigma_1',\beta_1'}\to \X_{\Sigma_0',\beta_0'}\times \X_{\Sigma_1',\beta_1'}$. Base changing $\X_{(\Phi_0,\phi_0)}\times \X_{(\Phi_1,\phi_1)}$ by this morphism, we get $\X_{(\Phi_1,\phi_1)}$, so $\X_{(\Phi_1,\phi_1)}$ has property $\mathbf P$. Similarly, $\X_{(\Phi_0,\phi_0)}$ has property $\mathbf P$.
\end{proof}

\section{Fantastacks: Easy-to-Draw Examples}\label{sec:fantastacks}
In this section, we introduce a broad class of smooth toric stacks which are especially easy to handle because $N$ is a lattice and the fan on $L$ is induced by a fan on $N$.

\begin{definition}\label{def:fantastack}
 Let $\Sigma$ be a fan on a lattice $N$, and let $\beta\colon \ZZ^n\to N$ be a homomorphism with finite cokernel so that every ray of $\Sigma$ contains some $\beta(e_i)$ and every $\beta(e_i)$ lies in the support of $\Sigma$. For a cone $\sigma\in \Sigma$, let $\hat \sigma=\mathrm{cone}(\{e_i|\beta(e_i)\in \sigma\})$. We define the fan $\hhat\Sigma$ on $\ZZ^n$ as the fan generated by all the $\hat\sigma$. We define $\F_{\Sigma,\beta} = \X_{\hhat\Sigma,\beta}$. Any toric stack isomorphic to some $\F_{\Sigma,\beta}$ is called a \emph{fantastack}.
\end{definition}
\begin{remark}\label{rmk:computing-X_hatSigma}
 The cones of $\hhat\Sigma$ are indexed by sets $\{e_{i_1},\dots, e_{i_k}\}$ such that $\{\beta(e_{i_1}),\dots, \beta(e_{i_k})\}$ is contained in a single cone of $\Sigma$. It is therefore easy to identify which open subvariety of $\AA^n$ is represented by $\hhat \Sigma$. Explicitly, define the ideal
 \[
  J_\Sigma = \Bigl(\prod_{\beta(e_i)\not\in \sigma} x_i \Bigm| \sigma\in \Sigma\Bigr).
 \]
 Then $X_{\hhat\Sigma} = \AA^n\setminus V(J_\Sigma)$. Note, as in the Cox construction of a toric variety, that $J_\Sigma$ is generated by the monomials $\prod_{\beta(e_i)\not\in \sigma} x_i$, where $\sigma$ varies over \emph{maximal} cones of $\Sigma$.
\end{remark}
\begin{remark}\label{rmk:doing-fantastack-examples}
 Since $\beta$ is a homomorphism of lattices, $C(\beta)^*$ can be computed by simply dualizing $\beta$. Since $\beta$ is assumed to have finite cokernel, $G_\beta^0=0$, so $G_\beta = D(\cok \beta^*)$. If $f\colon \ZZ^n\to \cok \beta^*$ is the cokernel of $\beta^*$, with $g_i=f(e_i)$, then we have that $\F_{\Sigma,\beta} = \bigl[(\AA^n\setminus V(J_\Sigma))/_{\smat{g_1&\cdots&g_n}} D(\cok \beta^*)\bigr]$ using Notation \ref{not:weights}.
\end{remark}

\begin{remark}
 Fantastacks are precisely the toric Artin stacks in \cite{can} which have trivial generic stabilizer.
\end{remark}
\begin{remark}
 The fantastack $\F_{\Sigma,\beta}$ has the toric variety $X_\Sigma$ as its good moduli space, as we will show in Example \ref{ex:fantastack-gms} (this is also proved in \cite[Theorem 5.5]{can}). In fact, a smooth toric stack $\X$ is a fantastack if and only if it has a toric variety $X$ as a good moduli space and the morphism $\X\to X$ restricts to an isomorphism of tori.
\end{remark}

\begin{example}\label{eg:fantastack-A^n/G_m^n}
 Let $\Sigma$ be the trivial fan on $N=0$. Let $\beta\colon \ZZ^n\to N$ be the zero map. Then $\hhat\Sigma$ is the fan of $\AA^n$, and $G_\beta=\GG_m^n$, so $\F_{\Sigma,\beta}=[\AA^n/\GG_m^n]$.
\end{example}
\begin{example}\label{eg:fantastack-smooth-toric-varieties}
 By Remark \ref{rmk:smooth-toric-stacks-are-fantastacks}, any smooth toric variety is a fantastack. If $X_\Sigma$ is a smooth toric variety, where $\Sigma$ is a fan on a lattice $N$, we construct $\beta\colon \ZZ^n\to N$ by sending the generators of $\ZZ^n$ to the first lattice points along the rays of $\Sigma$. Then $X_\Sigma\cong \F_{\Sigma,\beta}$.

 For a general (non-smooth) fan $\Sigma$, one can still construct $\beta$ as above, but the resulting fantastack $\F_{\Sigma,\beta}$ is not isomorphic to $X_\Sigma$. However, it is the \emph{canonical stack} over $X_\Sigma$, a sort of minimal stacky resolution of singularities (see Section \ref{sec:canonical-stacks}).
\end{example}

\begin{notation}\label{not:fantastacks}
 When describing fantastacks, we draw the fan $\Sigma$ and label $\beta(e_i)$ with the number $i$.
\end{notation}
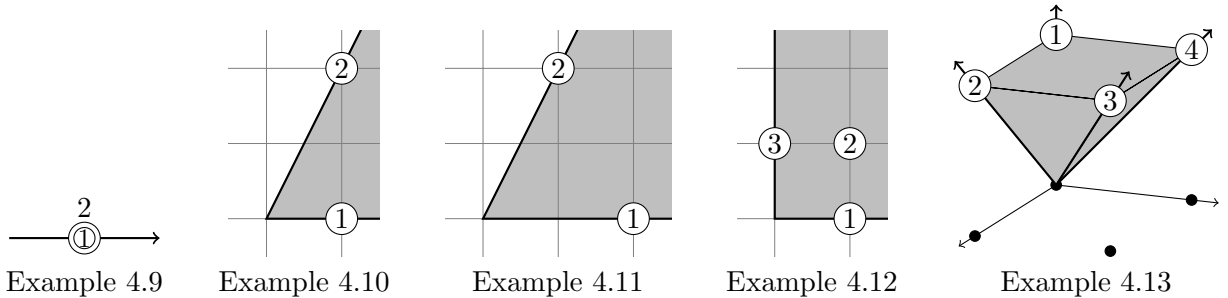
\begin{center}
\begin{tabular}{@{\extracolsep{1em}}ccccc}
  $\begin{tikzpicture}
   \draw [thick,->] (0,0) -- (2,0);
 \filldraw[draw=black,fill=white] (1,0) node[above=4pt] {$2$} circle (6pt)
   (1,0) node {\small $1$} circle (4pt);
 \end{tikzpicture}$
 &
 $\begin{tikzpicture}
 \clip (-.5,-.5) rectangle (1.5,2.5);
 \foreach \x/\y/\z/\w in {9/0/4/8}
   \filldraw [draw=black,fill=lightgray] (\x,\y) -- (0,0) -- (\z,\w);
 \draw[help lines] (-10,-10) grid (10,10);
 \foreach \x/\y/\z/\w in {9/0/4/8}
   \draw [thick] (\x,\y) -- (0,0) -- (\z,\w);
 \foreach \x/\y/\dottext in {1/0/1, 1/2/2}
   \filldraw[draw=black,fill=white] (\x,\y) node {$\dottext$} circle (6pt);
 \end{tikzpicture}$
 &
 $\begin{tikzpicture}
 \clip (-.5,-.5) rectangle (2.5,2.5);
 \foreach \x/\y/\z/\w in {9/0/4/8}
   \filldraw [draw=black,fill=lightgray] (\x,\y) -- (0,0) -- (\z,\w);
 \draw[help lines] (-10,-10) grid (10,10);
 \foreach \x/\y/\z/\w in {9/0/4/8}
   \draw [thick] (\x,\y) -- (0,0) -- (\z,\w);
 \foreach \x/\y/\dottext in {2/0/1, 1/2/2}
   \filldraw[draw=black,fill=white] (\x,\y) node {$\dottext$} circle (6pt);
 \end{tikzpicture}$
 &
 $\begin{tikzpicture}
 \clip (-.5,-.5) rectangle (1.5,2.5);
 \foreach \x/\y/\z/\w in {9/0/0/9}
   \filldraw [draw=black,fill=lightgray] (\x,\y) -- (0,0) -- (\z,\w);
 \draw[help lines] (-10,-10) grid (10,10);
 \foreach \x/\y/\z/\w in {9/0/0/9}
   \draw [thick] (\x,\y) -- (0,0) -- (\z,\w);
 \foreach \x/\y/\dottext in {1/0/1, 1/1/2, 0/1/3}
   \filldraw[draw=black,fill=white] (\x,\y) node {$\dottext$} circle (6pt);
 \end{tikzpicture}$
 &
 $\begin{tikzpicture}
    \pgfsetxvec{\pgfpointxy{.9}{-.1}}
    \pgfsetzvec{\pgfpointxy{-.6}{-.4}}
    \draw plot[only marks,mark=*] coordinates{(0,0,0) (2,0,0) (2,0,2) (0,0,2)};
    \draw [->] (0,0) -- (2.4,0,0);
    \draw [->] (0,0) -- (0,0,2.4);
    \draw [->,thick] (0,0,0) -- (0,2.4,0);
    \filldraw [draw=black,fill=lightgray]
      (0,0,0) -- (0,2,2) -- (2,2,2) -- cycle
      (0,0,0) -- (2,2,2) -- (2,2,0) -- cycle
      (0,2,0) -- (2,2,0) -- (2,2,2) -- (0,2,2) -- cycle;
    \draw [->,thick] (0,0,0) -- (2.3,2.3,0);
    \draw [->,thick] (0,0,0) -- (2.7,2.7,2.7);
    \draw [->,thick] (0,0,0) -- (0  ,2.5,2.5);
   \foreach \x/\y/\z/\dottext in {0/2/0/1, 0/2/2/2, 2/2/2/3, 2/2/0/4}
     \filldraw[draw=black,fill=white] (\x,\y,\z) node {$\dottext$} circle (6pt);
 \end{tikzpicture}$
 \\
 Example \ref{eg:fantastack-double-origin}&
 Example \ref{eg:fantastack-A_1}&
 Example \ref{eg:fantastack-A_1-rooted}&
 Example \ref{eg:fantastack-blowup-A^2}&
 Example \ref{eg:fantastack-square-cone}
\end{tabular}
\end{center}
\begin{example}\label{eg:fantastack-double-origin}
 Since a single cone contains all the $\beta(e_i)$, we have that $X_{\hhat\Sigma} = \AA^2$ (see Remark \ref{rmk:computing-X_hatSigma}). We have $\beta=(1\ 1)\colon \ZZ^2\to \ZZ$, and the cokernel of $\beta^*$ is
 \[
  \ZZ\xrightarrow{\beta^*=\smat{1\\ 1}} \ZZ^2\xrightarrow{\smat{1&\ng}} \ZZ.
 \]
 We see that $\F_{\Sigma,\beta}=[\AA^2/_{\smat{1&\ng}}\GG_m]$ (see Remark \ref{rmk:doing-fantastack-examples}). Note that this example contains the non-separated line $[(\AA^2\setminus 0)/_{\smat{1&\ng}}\GG_m]$ (Example \ref{eg:double-origin}) as an open substack.
\end{example}

\begin{example}\label{eg:fantastack-A_1}
 Since a single cone contains all the $\beta(e_i)$, we have that $X_{\hhat \Sigma}=\AA^2$. The cokernel of $\beta^*$ is
 \[
  \ZZ^2\xrightarrow{\beta^*=\smat{1&0\\ 1&2}} \ZZ^2\xrightarrow{\smat{1&1}} \ZZ/2.
 \]
 Note that $\beta^*$ can be obtained from the picture directly: the rows of $\beta^*$ are simply the coordinates of the $\beta(e_i)$.

 Therefore, $\F_{\Sigma,\beta} = [\AA^2/_{\smat{1&1}}\mu_2]$. This is a ``stacky resolution'' of the $A_1$ singularity
 \begin{align*}
  \AA^2/_{\smat{1&1}}\mu_2 &= \spec(k[x_1,x_2]^{\mu_2})\\
  &= \spec k[x_1^2,x_1x_2,x_2^2] = \spec \bigl(k[x,y,z]/(xy-z^2)\bigr).\qedhere
 \end{align*}
\end{example}

\begin{example}\label{eg:fantastack-A_1-rooted}
 As in the previous examples, $X_{\hhat \Sigma}=\AA^2$. The cokernel of $\beta^*$ is
 \[
  \ZZ^2\xrightarrow{\beta^*=\smat{2&0\\ 1&2}} \ZZ^2\xrightarrow{\smat{1&2}} \ZZ/4.
 \]
  Therefore, $\F_{\Sigma,\beta} = [\AA^2/_{\smat{1&2}}\mu_4]$. Like the previous example, this is a stacky resolution of the $A_1$ singularity
 \[
   \AA^2/_{\smat{1&2}}\mu_4 = \spec(k[x_1,t]^{\mu_4}) = \spec k[x_1^2,x_1t^2,t^4] = \spec \bigl(k[x,y,z]/(xy-z^2)\bigr).
 \]
 The coordinate $t$ in this example can be thought of as a formal square root of the coordinate $x_2$ in the previous example. In general, the moduli interpretation of smooth toric stacks presented in Section \ref{sec:moduli} shows that replacing $\beta(e_i)$ by a multiple corresponds to performing a root stack construction of the corresponding divisor (see for example \cite[\S1.3]{fmn} for a discussion of root constructions).
\end{example}

\begin{example}\label{eg:fantastack-blowup-A^2}
 We have $X_{\hhat\Sigma}=\AA^3$. The cokernel of $\beta^*$ is
 \[
  \ZZ^2\xrightarrow{\beta^*=\smat{1&0\\ 1&1\\ 0&1}}\ZZ^3\xrightarrow{\smat{1&-1&1}} \ZZ.
 \]
 So $\F_{\Sigma,\beta}=[\AA^3/_{\smat{1&\ng&1}}\GG_m]$.

 Note that refining the fan yields an \emph{open} substack. In this example, consider what happens when we refine the fan $\Sigma$ to the fan $\Sigma'$ below:
 
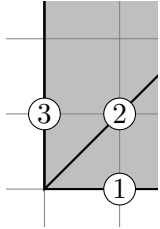
\[\begin{tikzpicture}
 \clip (-.5,-.5) rectangle (1.5,2.5);
 \foreach \x/\y/\z/\w in {9/0/9/9, 9/9/0/9}
   \filldraw [draw=black,fill=lightgray] (\x,\y) -- (0,0) -- (\z,\w);
 \draw[help lines] (-10,-10) grid (10,10);
 \foreach \x/\y/\z/\w in {9/0/0/9, 9/9/0/9}
   \draw [thick] (\x,\y) -- (0,0) -- (\z,\w);
 \foreach \x/\y/\dottext in {1/0/1, 1/1/2, 0/1/3}
   \filldraw[draw=black,fill=white] (\x,\y) node {$\dottext$} circle (6pt);
 \end{tikzpicture}\]
 Here $G_\beta$ is unchanged; indeed, $G_\beta$ depends only on $\beta$, not on $\Sigma$. However, we remove $\mathrm{cone}(e_1,e_3)$ from $\hhat\Sigma'$. The resulting stack is therefore the open substack $\F_{\Sigma',\beta}=[(\AA^3\setminus V(x_1,x_3))/_{\smat{1&-1&1}}\GG_m]$, which is the blowup of $\AA^2$ at the origin (cf.~Example \ref{eg:fantastack-smooth-toric-varieties}).

 The birational transformation $Bl_0(\AA^2)\to \AA^2$ can therefore be realized as the morphism of good moduli spaces induced by the open immersion $\F_{\Sigma',\beta}\to \F_{\Sigma,\beta}$.
\end{example}
\begin{example}\label{eg:fantastack-square-cone}
 We have that $X_{\hhat \Sigma}=\AA^4$. The cokernel of $\beta^*$ is
 \[
  \ZZ^3\xrightarrow{\beta^*=\smat{0&0&1\\ 1&0&1\\ 1&1&1\\ 0&1&1}} \ZZ^4\xrightarrow{\smat{1&\ng&1&\ng}} \ZZ,
 \]
 so $\F_{\Sigma,\beta}=[\AA^4/_{\smat{1&\ng&1&\ng}}\GG_m]$. We will see in Section \ref{sec:canonical-stacks} that $\F_{\Sigma,\beta}$ is the canonical stack over the singular toric variety
 \begin{align*}
  X_\Sigma = \spec k[x_1,x_2,x_3,x_4]^{\GG_m} &= \spec k[x_1x_2,x_3x_4,x_1x_4,x_2x_3] \\
  &= \spec \bigl(k[x,y,z,w]/(xy-zw)\bigr).
 \end{align*}
 It can be regarded as a ``stacky resolution'' of the singularity.

 Note that the two standard toric small resolutions of this singularity are both open substacks of this stacky resolution.
 
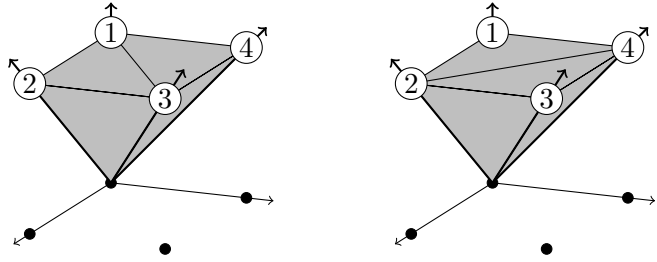
\[\begin{tikzpicture}
    \pgfsetxvec{\pgfpointxy{.9}{-.1}}
    \pgfsetzvec{\pgfpointxy{-.6}{-.4}}
    \draw plot[only marks,mark=*] coordinates{(0,0,0) (2,0,0) (2,0,2) (0,0,2)};
    \draw [->] (0,0) -- (2.4,0,0);
    \draw [->] (0,0) -- (0,0,2.4);
    \draw [->,thick] (0,0,0) -- (0,2.4,0);
    \filldraw [draw=black,fill=lightgray]
      (0,0,0) -- (0,2,2) -- (2,2,2) -- cycle
      (0,0,0) -- (2,2,2) -- (2,2,0) -- cycle
      (0,2,0) -- (2,2,0) -- (2,2,2) -- (0,2,2) -- cycle
      (0,2,0) -- (2,2,2);
    \draw [->,thick] (0,0,0) -- (2.3,2.3,0);
    \draw [->,thick] (0,0,0) -- (2.7,2.7,2.7);
    \draw [->,thick] (0,0,0) -- (0  ,2.5,2.5);
   \foreach \x/\y/\z/\dottext in {0/2/0/1, 0/2/2/2, 2/2/2/3, 2/2/0/4}
     \filldraw[draw=black,fill=white] (\x,\y,\z) node {$\dottext$} circle (6pt);
 \end{tikzpicture}
 \qquad\qquad
 \begin{tikzpicture}
    \pgfsetxvec{\pgfpointxy{.9}{-.1}}
    \pgfsetzvec{\pgfpointxy{-.6}{-.4}}
    \draw plot[only marks,mark=*] coordinates{(0,0,0) (2,0,0) (2,0,2) (0,0,2)};
    \draw [->] (0,0) -- (2.4,0,0);
    \draw [->] (0,0) -- (0,0,2.4);
    \draw [->,thick] (0,0,0) -- (0,2.4,0);
    \filldraw [draw=black,fill=lightgray]
      (0,0,0) -- (0,2,2) -- (2,2,2) -- cycle
      (0,0,0) -- (2,2,2) -- (2,2,0) -- cycle
      (0,2,0) -- (2,2,0) -- (2,2,2) -- (0,2,2) -- cycle
      (2,2,0) -- (0,2,2);
    \draw [->,thick] (0,0,0) -- (2.3,2.3,0);
    \draw [->,thick] (0,0,0) -- (2.7,2.7,2.7);
    \draw [->,thick] (0,0,0) -- (0  ,2.5,2.5);
   \foreach \x/\y/\z/\dottext in {0/2/0/1, 0/2/2/2, 2/2/2/3, 2/2/0/4}
     \filldraw[draw=black,fill=white] (\x,\y,\z) node {$\dottext$} circle (6pt);
 \end{tikzpicture}\]
 The above fans correspond to $[(\AA^4\setminus V(x_2,x_4))/_{\smat{1&\ng&1&\ng}}\GG_m]$ and $[(\AA^4\setminus V(x_1,x_3))/_{\smat{1&\ng&1&\ng}}\GG_m]$. These are both toric varieties (cf.~Example \ref{eg:fantastack-smooth-toric-varieties}).
\end{example}

\subsection{Some Non-fantastack Examples}

\begin{example}
 Suppose $\{n_1,\dots, n_k\}$ is a set of positive integers. Let $N$ be $\ZZ^r\oplus \bigoplus_{i=1}^k (\ZZ/n_i\ZZ)$, $L$ be $0$, $\Sigma$ be the trivial fan on $L$, and $\beta\colon L\to N$ the zero map.

 To compute $G_\beta$, we take a free resolution of $C(\beta)$, namely
 \[
  \ZZ^k\xrightarrow{\mathrm{diag}(n_1,\dots, n_k)\oplus 0}\ZZ^k\oplus \ZZ^r.
 \]
 We see that
 \[
  H^0(C(\beta)^*) = D(G_\beta^{0})=\ZZ^r,\qquad
  H^1(C(\beta)^*) = D(G_\beta^{1})=\bigoplus_{i=1}^k (\ZZ/n_i\ZZ).
 \]
 Therefore, $G_\beta=\GG_m^r\times \prod \mu_{n_i}$. Since $X_\Sigma=\spec k$, we have that $\X_{\Sigma,\beta}=BG_\beta$.
\end{example}

\begin{example}[{Cf.~\cite[Examples 2.1 and 3.5]{bcs}}]\label{eg:bar(M_1,1)}
 Consider the stacky fan in which $L=\ZZ^2$, $\Sigma = \raisebox{-2pt}{\tikz[scale=.5] \draw[<->] (1,0) -- (0,0) -- (0,1);}$ is the fan corresponding to $\AA^2\setminus \{0\}$, $N=\ZZ\oplus (\ZZ/2)$, and $\beta = \smat{2&\text-3 \\ 1&0}\colon \ZZ^2\to \ZZ\oplus (\ZZ/2)$. Then $C(\beta)^*$ is represented by the map $\smat{2& 1\\ \text- 3& 0\\ 0&2}\colon \ZZ^2\to \ZZ^3$. This map is injective, and its cokernel is $(6\ 4\ \text-3)\colon \ZZ^3\to \ZZ$. Therefore, $G_\beta = \GG_m$, and the induced map to $T_L=\GG_m^2$ is given by $t\mapsto (t^6,t^4)$. So $\X_{\Sigma,\beta} = [(\AA^2\setminus \{0\})/_{\smat{6& 4}}\GG_m]$. This is the weighted projective stack $\PP(6,4)$, the moduli stack of elliptic curves $\bbar{\mathcal M}_{1,1}$.
\end{example}

We repeat the previous example to illustrate that it can be realized as a closed substack of a fantastack. This approach is explained in Remark \ref{rmk:non-strict-is-closed-substack}.
\begin{example}\label{eg:non-strict-as-closed-substack}
 Consider the stacky fan in which $L=\ZZ^2$, $\Sigma = \raisebox{-2pt}{\tikz \draw[<->] (.4,0) -- (0,0) -- (0,.4);}$ is the fan corresponding to $\AA^2\setminus \{0\}$, $N=\ZZ\oplus (\ZZ/2)$, and $\beta = \smat{2&\text-3 \\ 1&0}\colon \ZZ^2\to \ZZ\oplus (\ZZ/2)$.

 We replace $\beta$ by the quasi-isomorphic map $\beta' = \smat{2&-3&0\\ \text1&0&2}\colon \ZZ^3\to \ZZ^2$ and the fan $\Sigma$ by the fan $\Sigma'$ obtained by adding the cone $\tau$.
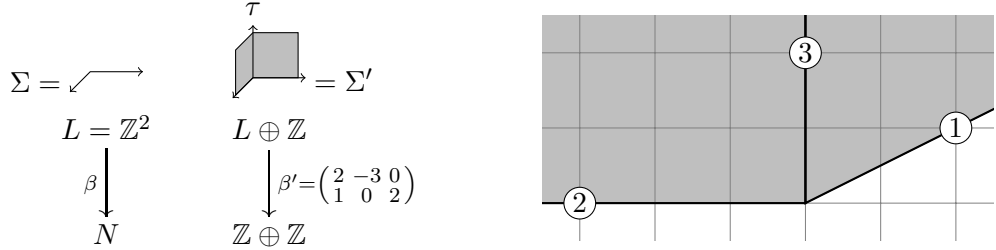
 \[\raisebox{3.5em}{$\xymatrix{
  L=\ZZ^2\ar[d]_\beta \POS p+(0,.7) *+{\llap{$\Sigma=\;$}\tikz \draw[<->] (0,0,.7) -- (0,0,0) -- (.7,0,0);} & L\oplus \ZZ\ar[d]^{\beta'=\smat{2&-3&0\\ \text1&0&2}} \POS p+(0,1.2) *+{\begin{tikzpicture}
    \filldraw[fill=lightgray]
      (0,0,0) -- (.6,0,0) -- (.6,.6,0) -- (0,.6,0) --
      (0,.6,.6) -- (0,0,.6) -- cycle;
    \draw[->] (0,0,0) -- (.7,0,0);
    \draw[->] (0,0,0) -- (0,0,.7);
    \draw[->] (0,0,0) -- (0,.7,0) node[anchor=south] {$\tau$};
  \end{tikzpicture}\raisebox{1mm}{\rlap{$\;=\Sigma'$}}}\\
  N & \ZZ\oplus \ZZ
 }$}\qquad\qquad
 \begin{tikzpicture}
   \clip (-3.5,-.5) rectangle (2.5,2.5);
   \foreach \x/\y/\z/\w in {8/4/0/8, 0/8/-8/0}
     \filldraw [draw=black,fill=lightgray] (\x,\y) -- (0,0) -- (\z,\w);
   \draw[help lines] (-10,-10) grid (10,10);
   \foreach \x/\y/\z/\w in {8/4/0/8, 0/8/-8/0}
     \draw [thick] (\x,\y) -- (0,0) -- (\z,\w);
   \foreach \x/\y/\dottext in {2/1/1, -3/0/2, 0/2/3}
     \filldraw[draw=black,fill=white] (\x,\y) node {$\dottext$} circle (6pt);
 \end{tikzpicture}
 \]
 We see that $\X_{\Sigma',\beta'}$ is the fantastack corresponding to the fan on the right. Explicitly, it is the fantastack $[(\AA^3\setminus V(x_1,x_2))/_{\smat{6&4&\text-3}}\GG_m]$. The closed substack $\X_{\Sigma,\beta}$ is the divisor corresponding to the ``extra ray'', which is numbered 3 in the picture. That is, it is the divisor $V(x_3) = [(\AA^2\setminus V(x_1,x_2))/_{\smat{6&4}}\GG_m]$.
\end{example}

\section{Canonical Stacks}\label{sec:canonical-stacks}

Given a non-strict toric stack, there is a canonical smooth non-strict toric stack of which it is a good moduli space. The purpose of this section is to construct and characterize this canonical smooth stack.

Given a fan $\Sigma$ on a lattice $L$, the Cox construction \cite[\S 5.1]{cls} of the toric variety $X_\Sigma$ produces an open subscheme $U$ of $\AA^n$ and a subgroup $H\subseteq \GG_m^n$ so that $X_\Sigma=U/H$. That is, $[U/H]\to X_\Sigma$ is a good moduli space morphism in the sense of \cite{Alper:good}. We recall and generalize this construction here.

Let $(\Sigma,\beta)$ be a non-strict stacky fan. Let $\Sigma(1)$ be the set of rays of $\Sigma$. Let $M\subseteq L$ be the saturated sublattice spanned by $\Sigma$, and let $M'\subseteq L$ be a direct complement to $M$. For each ray $\rho\in \Sigma(1)$, let $u_\rho$ be the first element of $M$ along $\rho$, and let $e_\rho$ be the generator in $\ZZ^{\Sigma(1)}$ corresponding to $\rho$. We then have a morphism $\Phi\colon \ZZ^{\Sigma(1)}\times M'\to L$ given by $(e_\rho,m)\mapsto u_\rho+m$. We define a fan $\ttilde\Sigma$ on $\ZZ^{\Sigma(1)}\times M'$. For each $\sigma\in \Sigma$, we define $\tilde\sigma\in \ttilde\Sigma$ as the cone generated by $\{e_\rho|\rho\in \sigma\}$. The morphism of non-strict stacky fans
\[\xymatrix{
 \ZZ^{\Sigma(1)}\times M' \ar[r]^-\Phi\ar[d]_{\tilde\beta} \POS p+(0,.7) *+{\ttilde\Sigma}="tsig" & L\ar[d]^{\beta} \POS p+(0,.7) *+{\Sigma} \ar@{<-} "tsig"\\
 N\ar@{=}[r] & N
}\]
induces a toric morphism $\X_{\ttilde\Sigma,\tilde\beta}\to\X_{\Sigma,\beta}$.

\begin{definition}\label{def:canonical-stack}
 We call $\X_{\ttilde\Sigma,\tilde\beta}$ the \emph{canonical stack} over $\X_{\Sigma,\beta}$, and we say that the morphism $\X_{\ttilde\Sigma,\tilde\beta}\to \X_{\Sigma,\beta}$ is a \emph{canonical stack morphism}.
\end{definition}

\begin{remark}\label{rmk:usual-cox-mod-G_beta}
 The Cox construction expresses $X_\Sigma$ as a quotient of $X_{\ttilde\Sigma}$ by $G_\Phi$. Applying Lemma \ref{lem:triangle=>ses-of-G_betas} ($\Phi$ has finite cokernel by construction), we see that the morphism constructed above is obtained by quotienting the morphism $[X_{\ttilde\Sigma}/G_\Phi]\to X_\Sigma$ by the action of $G_\beta$. This shows that the construction above commutes with quotienting $\X_{\Sigma,\beta}$ by its torus (i.e.~replacing $G_\beta$ by $G_{L\to 0}$).
\end{remark}

The remainder of this subsection is dedicated to justifying this terminology by showing that the canonical stack has a universal property (Proposition \ref{prop:canonical-stack}). Notably, this universal property shows that the canonical stack depends only on the stack $\X_{\Sigma,\beta}$ and its torus action, not on the stacky fan $(\Sigma,\beta)$.

Recall (Definition \ref{def:cohomologically-affine}) that $\X_{\Sigma,\beta}$ is \emph{cohomologically affine} if $X_\Sigma$ is affine. As shown in Remark \ref{rmk:coh-affine-definitions-agree}, this property depends only on $\X_{\Sigma,\beta}$ and not on the stacky fan $(\Sigma,\beta)$.

\begin{lemma}\label{lem:canonical-stack-affine}
 Let $\X_{\Sigma,\beta}$ be a cohomologically affine toric stack with $n$ torus-invariant irreducible divisors. Suppose $f\colon \X_{\Sigma',\beta'}\to \X_{\Sigma,\beta}$ is a toric surjection from a smooth cohomologically affine toric stack with $n$ torus-invariant divisors, which restricts to an isomorphism on tori. Then $\X_{\Sigma',\beta'}\to \X_{\Sigma,\beta}$ factors uniquely through the canonical stack over $\X_{\Sigma,\beta}$.
\end{lemma}
\begin{proof}
 By Theorem \ref{thm:morphisms-come-from-fans}, we may assume $f$ is induced by a morphism of stacky fans $(\Phi,\phi)\colon (\Sigma',\beta')\to (\Sigma,\beta)$. Since $f$ restricts to an isomorphism on tori, $\phi$ must be an isomorphism.

 Since we are only considering torus-equivariant morphisms, we may verify the property after quotienting by the action of the torus, so we may assume $\X_{\Sigma,\beta}$ is the quotient of an affine toric variety by its torus, and $\X_{\Sigma',\beta'}$ is the quotient of a smooth affine toric variety with $n$ divisors by its torus. This identifies $\X_{\Sigma',\beta'}$ as $[\AA^n/\GG_m^n]$, so we may assume $\Sigma'$ is the fan of $\AA^n$. We identify the first lattice points along the rays of $\Sigma'$ with the generators $e_i\in \ZZ^n$.

 Since $f$ is surjective, the induced morphism $\Sigma'\to \Sigma$ is surjective by Lemma \ref{lem:coh-affine-surjection}. Every ray of $\Sigma$ is then the image of a unique ray of $\Sigma'$, since $\Sigma'$ has only $n$ rays. Suppose $\Phi(e_i)=k_i \rho_i$. Then we see that $\Phi$ factors uniquely through the canonical stack via the morphism of fans $\Sigma'\to \ttilde\Sigma$ given by sending $e_i$ to $k_i e_{\rho_i}$.
\end{proof}

\begin{remark}[``Canonical stacks are stable under base change by open immersions'']\label{rmk:canonical-stack-open-immersions}
 The pre-image of a torus-invariant divisor of $\X_{\Sigma,\beta}$ is a divisor in its canonical stack. So by Lemma \ref{lem:canonical-stack-affine}, restricting a canonical stack morphism to the open complement of a torus-invariant divisor yields a canonical stack morphism.
\end{remark}

As a corollary, we get the following result.

\begin{proposition}[Universal property of the canonical stack]\label{prop:canonical-stack}
 Suppose $\X_{\Sigma',\beta'}\to \X_{\Sigma,\beta}$ is a toric morphism from a smooth toric stack, which restricts to an isomorphism of tori, and which restricts to a canonical stack morphism over every torus-invariant cohomologically affine open substack of $\X_{\Sigma,\beta}$. Then $\X_{\Sigma',\beta'}\to\X_{\Sigma,\beta}$ is a canonical stack morphism.
\end{proposition}

By Remark \ref{rmk:smooth-toric-stacks-are-fantastacks}, we see that the canonical stack morphism over a smooth toric stack is an isomorphism. In particular, this shows that for any non-strict toric stack $\X_{\Sigma,\beta}$, the canonical stack is isomorphic to $\X_{\Sigma,\beta}$ over its smooth locus. Thus the canonical stack can be regarded as a (canonical!) ``stacky resolution of singularities'' (cf.~Examples \ref{eg:fantastack-A_1} and \ref{eg:fantastack-square-cone}).

By Remark \ref{rmk:canonical-stack-open-immersions}, the definition of a canonical stack morphism can be extended to stacks which are only locally known to be toric stacks. This will be important in \cite{toricartin2}, where we prove that certain stacks are toric by showing that they are locally toric, that their canonical stacks are (globally) toric, and that the property of being toric can be ``descended'' along canonical stack morphisms.

\begin{definition}\label{def:general-canonical-stack}
 Suppose $\X$ is a stack with an open cover by non-strict toric stacks with a common torus. A morphism from a smooth stack $\Y\to \X$ is a \emph{canonical stack morphism} if it restricts to canonical stack morphisms on the open toric substacks of $\X$.
\end{definition}

\section{Toric Good Moduli Space Morphisms}\label{sec:gms-morphisms}

\subsection{Good Moduli Space Morphisms}

In \cite{Alper:good}, Alper introduces the notion of a good moduli space morphism, which generalizes the notion of a good quotient \cite{git} and is moreover a common generalization of the notion of a tame Artin stack \cite[Definition 3.1]{AOV} and of a coarse moduli space \cite[Theorem 4.10]{FaltingsC:degenerationOfAbelian}. Our goal in this section is to prove Theorem \ref{thm:main-gms}, which characterizes toric good moduli space morphisms.

\begin{definition}\label{def:gms}
 A quasi-compact and quasi-separated morphism of algebraic stacks $f\colon \X\to \Y$ is a \emph{good moduli space morphism} if
 \begin{itemize}
  \item ($f$ is \emph{Stein}) the morphism $\O_\Y\to f_*\O_\X$ is an isomorphism, and
  \item ($f$ is \emph{cohomologically affine}) the pushforward functor $f_*\colon \qcoh(\O_\X)\to \qcoh(\O_\Y)$ is exact.
 \end{itemize}
\end{definition}

\begin{definition}\label{def:unstable}
 We say that a cone $\tau$ of a non-strict stacky fan $(\Sigma,\beta\colon L\to N)$ is \emph{unstable} if any of the following equivalent conditions are satisfied:
 \begin{itemize}
  \item every linear functional $N\to \ZZ$ which is non-negative on $\beta(\tau)$ vanishes on $\beta(\tau)$,
  \item the relative interior of the image of $\tau$ in the lattice $N/N_\tor$ contains $0$, or
  \item $\tau\cap \ker\beta$ is not contained in any proper face of $\tau$.
 \end{itemize}
\end{definition}

As in Appendix \ref{sec:non-uniqueness-of-fans}, for a morphism of fans $\Phi\colon \Sigma\to \Sigma'$, the pre-image of a cone $\Phi^{-1}(\sigma')$ refers to the subfan of cones which are mapped into $\sigma'$, and we say that this pre-image ``is a single cone $\sigma\in \Sigma$'' if it is the fan consisting of $\sigma$ and its faces.
\begin{theorem}
\label{thm:main-gms}
 Let $(\Phi,\phi)\colon (\Sigma,\beta\colon L\to N)\to (\Sigma',\beta'\colon L'\to N')$ be a morphism of non-strict stacky fans, where $\beta$ has finite cokernel. The induced morphism $\X_{\Sigma,\beta}\to\X_{\Sigma',\beta'}$ is a good moduli space morphism if and only if
 \begin{enumerate}
  \item For every $\sigma'\in \Sigma'$, the pre-image $\Phi^{-1}(\sigma')$ is a single cone $\sigma\in \Sigma$ with $\Phi(\sigma)=\sigma'$.\\
  In particular, the pre-image of the zero cone is some cone $\tau\in \Sigma$, and the pre-image of any other cone has $\tau$ as a face.  Let $\tau^{gp}$ denote $(L\cap\tau)^{gp}$.
  \item $\tau$ is unstable,
  \item $\phi$ is surjective, and
  \item $(\ker\phi)/\beta(\tau^\gp)$ is finite.
 \end{enumerate}
\end{theorem}

As a corollary of Theorem \ref{thm:main-gms}, we obtain a combinatorial criterion for when a toric stack has a toric variety as a good moduli space. This is done by explicitly constructing a fan $(\Sigma',\id_{N'})$ for the good moduli space.
\begin{notation}\label{not:fan-of-gms}
  Suppose $(\Sigma,\beta\colon L\to N)$ is a non-strict stacky fan with $\cok\beta$ finite and with the property that among unstable cones of $\Sigma$, there is a unique maximal one $\tau$. Let $N'=N/\sat_N(\beta(\tau^\gp))$, let $\phi\colon N\to N'$ be the quotient map, and let $\Phi=\phi\circ \beta$. Let $\Sigma'$ be the set of cones $\sigma'$ on $N'$ such that
   \begin{enumerate}
    \item[(a)] $\Phi^{-1}(\sigma')$ is a single cone, and
    \item[(b)] $\sigma' = \Phi(\Phi^{-1}(\sigma))$.
   \end{enumerate}
  This $\Sigma'$ is a fan.\footnote{Any face $\sigma''$ of $\sigma'\in \Sigma'$ is the vanishing locus of some linear functional $\lambda'\in \hom(N',\ZZ)$ which is non-negative on $\sigma'$. Then $\Phi^{-1}(\sigma'')$ is the vanishing locus of $\lambda'$ on $\Phi^{-1}(\sigma')$, so it is a face of $\Phi^{-1}(\sigma')$ and hence a single cone. It is clear that this cone must surject onto $\sigma''$, so $\Sigma'$ is closed under taking faces.

  If $\sigma',\sigma''\in \Sigma'$, then $\Phi^{-1}(\sigma')$ and $\Phi^{-1}(\sigma'')$ are cones in $\Sigma$, so their intersection is a common face. That is, there is some linear functional $\lambda\in\hom(N,\ZZ)$ which is non-negative on $\Phi^{-1}(\sigma')$ and non-positive on $\Phi^{-1}(\sigma'')$. Since $\tau$ is a common face of these two cones, $\lambda$ must vanish on $\tau^\gp$, so it must be induced by some $\lambda'\in \hom(N',\ZZ)$ (possibly after scaling). This $\lambda'$ is then non-negative on $\sigma'$ and non-positive on $\sigma''$, so the intersection of the two is a common face.}
\end{notation}

\begin{corollary}
\label{cor:main-gms-tv}
 Let $(\Sigma,\beta\colon L\to N)$ be a non-strict stacky fan with $\cok\beta$ finite. Then $\X_{\Sigma,\beta}$ has a variety as a good moduli space if and only if
 \begin{enumerate}
   \item[(i)] among unstable cones of $\Sigma$, there is a unique maximal one $\tau$, and
   \item[(ii)] using Notation \ref{not:fan-of-gms}, $\Phi$ induces a map of fans $(\Sigma,N)\to(\Sigma',N')$.
 \end{enumerate}
Moreover, if the conditions are satisfied, the good moduli space of $\X_{\Sigma,\beta}$ is the natural toric morphism to $X_{\Sigma'}$.
\end{corollary}

\begin{proof}
  If the two conditions are satisfied, then $X_{\Sigma'}=\X_{\Sigma',\id\colon N'\to N'}$ is a toric variety, and by Theorem \ref{thm:main-gms}, $\X_{\Phi,\phi}$ is a good moduli space morphism.

  Conversely, if $\X_{\Sigma,\beta}$ has a variety as a good moduli space, it must be a toric morphism to a toric variety by \cite[Proposition 7.2]{toricartin2}. Suppose $f\colon\X_{\Sigma,\beta}\to X_{\Sigma''}$ is a toric good moduli space morphism to a toric variety, where $\Sigma''$ is a fan on a lattice $N''$. By Theorem \ref{thm:morphisms-come-from-fans}, there are morphisms of stacky fans $(\Phi_0,\phi_0)\colon (\Sigma_0,\beta_0)\to (\Sigma,\beta)$ and $(\Phi'',\phi'')\colon (\Sigma_0,\beta_0)\to (\Sigma'',\id_{N''})$ so that $\X_{(\Phi_0,\phi_0)}$ is an isomorphism and $f=\X_{(\Phi'',\phi'')}\circ \X_{(\Phi_0,\phi_0)}^{-1}$. By Theorem \ref{thm:isomorphisms}, $(\Sigma,\beta)$ satisfies conditions (i) and (ii) if and only if $(\Sigma_0,\beta_0)$ does. We may therefore  assume that the good moduli space morphism is induced by a morphism of non-strict fans $(\Psi,\psi)\colon (\Sigma,\beta)\to (\Sigma'',\id_{N''})$.

  We now apply Theorem \ref{thm:main-gms} to verify (i): by (1) and (2), $\tau=\Psi^{-1}(0)$ is the unique maximal unstable cone; by (3) and (4), $\psi$ is surjective with kernel $\sat_N(\beta(\tau^\gp))$ (here we are using the fact that $N''$ is torsion-free). Therefore $\psi=\phi$ and $\Psi=\psi\circ\beta=\phi\circ\beta=\Phi$. Finally, by Theorem \ref{thm:main-gms}(1), condition (ii) is satisfied.
\end{proof}

\begin{warning}\label{warn:variety-gms}
  Corollary \ref{cor:main-gms-tv} provides a criterion for a non-strict toric stack to have a \emph{variety} as a good moduli space. However, there are toric stacks which have good moduli spaces that are not varieties: the non-separated line (Example \ref{eg:double-origin}) is a scheme, so is its own good moduli space, but it fails the criteria of the corollary (see Example \ref{eg:gms-non-separated-line}).
\end{warning}

\begin{remark}\label{rmk:alg-space-gms}
  A slight variation of Corollary \ref{cor:main-gms-tv} gives a criterion for a toric stack to have a good moduli space (which is necessarily a scheme by \cite[Proposition 7.2]{toricartin2}) and a construction of the good moduli space if it exists. We describe the construction and criterion here and leave the proof as an exercise for the reader.

  Given $(\Sigma,\beta\colon L\to N)$, suppose $\tau$ is the unique maximal unstable cone. Then $\tau$ is also the unique maximal unstable cone in the stacky fan $(\Sigma,\Phi\colon L\to L')$, where $L'=L/\tau^\gp$. Let $\Sigma'$ be the induced fan on $L'$, as described in Notation \ref{not:fan-of-gms} (where we use $\Phi$ in place of $\beta$). Let $N'=N/\sat_N(\beta(\tau^\gp))$, and let $\beta'\colon L'\to N'$ be the map induced by $\beta$.

  Consider the following diagram of toric monoids indexed by $\Sigma'$: for each $\sigma'\in\Sigma'$, we have the toric monoid $\beta'(\sigma')\cap N'$, and for every face relation in $\Sigma'$, we have the corresponding inclusion of monoids. This is a tight diagram of toric monoids (see \cite[\S2]{toricartin2}), so it is witnessed by a fan $\Sigma''$ on some colimit lattice $L''$, which has a map $\beta''\colon L''\to N'$ \cite[Corollary 2.12]{toricartin2}.

  A toric stack $\X_{\Sigma,\beta}$ has a good moduli space if and only if (i) $\Sigma$ has a unique maximal unstable cone $\tau$, and (ii) the morphism $\Phi$ described above induces a map of fans $\Sigma\to \Sigma'$. In this case, the good moduli space of $\X_{\Sigma,\beta}$ is the toric stack $\X_{\Sigma'',\beta''}$.
\end{remark}

\subsection{Proof of Theorem \ref{thm:main-gms}}
\label{subsec:local-gms}

We prove Theorem \ref{thm:main-gms} through several lemmas and propositions. 

\begin{lemma}\label{lem:product-of-gms-is-gms}
 For $i=0,1$, let $(\Phi_i,\phi_i)$ be a morphism of non-strict stacky fans. Then $(\Phi_0\times \Phi_1,\phi_0\times \phi_1)$ induces a good moduli space morphism of non-strict toric stacks if and only if each $(\Phi_i,\phi_i)$ does.
\end{lemma}
\begin{proof}
 Good moduli spaces are stable under composition (this follows quickly from the definition) and base change \cite[Proposition 4.7(i)]{Alper:good}, so the result follows from Proposition \ref{prop:product-of-property-P}.
\end{proof}

\begin{lemma}\label{l:gms-comp}
  If $f:\X\to\Y$ and $g:\Y\to\Z$ are morphisms of algebraic stacks and $f$ is a good moduli space morphism, then
  \begin{enumerate}
    \item\label{item:stein-composition} $g$ is Stein if and only if $g\circ f$ is, and
    \item\label{item:coh-affine-composition} $g$ is cohomologically affine if and only if $g\circ f$ is as well.
  \end{enumerate}
\end{lemma}
\begin{proof}
  (\ref{item:stein-composition}) Stein morphisms are clearly stable under composition. For the converse, note that $g_*\O_\Y=g_*f_*\O_\X=(g\circ f)_*\O_\X=\O_\Z$, so $g$ is Stein.

  (\ref{item:coh-affine-composition}) Cohomologically affine morphisms are stable under composition. For the converse, first note that for any quasi-coherent sheaf $\F$ on $\Y$, we have $f_*f^*\F=\F$ by \cite[Proposition 4.5]{Alper:good}. Since $Rf_*=f_*$ and $R(g\circ f)_*=(g\circ f)_*$, we have
  \[
   Rg_*\F=Rg_* Rf_*f^*\F=R(g\circ f)_*f^*\F=(g\circ f)_*f^*\F=g_*(f_*f^*\F)=g_*\F.\qedhere
  \]
\end{proof}

Throughout the rest of this subsection (until the proof of Theorem \ref{thm:main-gms}), we use the following setup. We have a morphism of non-strict stacky fans $(\Phi,\phi)\colon (\sigma,\beta\colon L\to N)\to (\sigma',\beta'\colon L'\to N)$, where $\sigma$ is a single cone on $L$, and $\sigma'=\Phi(\sigma)$:
\[\xymatrix{
 L \ar[r]^\Phi\ar[d]_\beta \POS p+(0,.7) *+{\sigma}="s" & L'\ar[d]^{\beta'} \POS p+(0,.7) *+{\sigma'} \ar@{<<-} "s"\\
 N\ar[r]^\phi & N'
}\]

\begin{lemma}[``Characterization when target is a toric variety'']\label{lem:quot-by-kernel-is-gms2}
  If $\beta$ has finite cokernel and $\beta'$ is an isomorphism, the induced map $\X_{\sigma,\beta}\to X_{\sigma'}$ is a good moduli space morphism if and only if every $\lambda\in \hom(N,\ZZ)$ which is non-negative on $\sigma$ (i.e.~$\lambda\circ\beta$ is non-negative on $\sigma$) is of the form $\phi^*(\lambda')$ for a unique $\lambda'\in \hom(N',\ZZ)$ which is non-negative on $\sigma'$.
\end{lemma}
\begin{proof}
 Since $X_\sigma=\spec k[\sigma^\vee\cap L^*]$ and $X_{\sigma'}=\spec k[\sigma'^\vee\cap L'^*]$ are affine and $G_\beta$ is linearly reductive, the induced map $\X_{(\Phi,\phi)}$ is a good moduli space morphism if and only if $k[\sigma'^\vee\cap L'^*]\to k[\sigma^\vee\cap L^*]$ is the inclusion of the ring of invariants $k[\sigma^\vee\cap L^*]^{G_\beta}= k[\sigma^\vee\cap L^*]^{G^1_\beta}$ (this last equality holds because $G^0_\beta$ acts trivially on $X_\sigma$). Letting $\pi\colon L^*\to D(G_\beta^1)$ be the map defining the action of $G_\beta^1$ on $T_L$, this says that $\X_{(\Phi,\phi)}$ is a good moduli space morphism if and only if $\sigma^\vee \cap \ker\pi =\Phi^*(\sigma'^\vee\cap L'^*)$, and $\Phi^*\colon \sigma'^\vee\cap L'^*\to \sigma^\vee\cap L^*$ is injective. Consider the following commutative diagram. Note that $\beta^*$ is injective since $\cok\beta$ is finite.
  \[\xymatrix{
    D(G_\beta^1) & L^* \ar[l]_-\pi \ar@{<-_)}[d]^{\beta^*} & L'^* \ar@{<-}[d]_\wr^{\beta'^*} \ar[l]_{\Phi^*}\\
    & N^* \ar[ul]^0  & N'^* \ar[l]_{\phi^*}
  }\]

  Suppose first that $\X_{(\Phi,\phi)}$ is a good moduli space morphism. Given $\lambda\in N^*$ which is non-negative on $\sigma$, we have that $\beta^*(\lambda)\in \sigma^\vee\cap \ker(\pi) = \Phi(\sigma'^\vee\cap L'^*)$ and $\Phi^*|_{\sigma'^\vee\cap L'^*}$ is injective, so there is a unique $\lambda'\in N'^*$ which is non-negative on $\sigma'$ so that $\Phi^*(\beta'^*(\lambda'))=\beta^*(\phi^*(\lambda'))=\beta^*(\lambda)$. Since $\beta^*$ is injective, this proves the desired result.

  For the converse, choose a free resolution
  \[
  0\to\ZZ^s\xrightarrow{Q}\ZZ^r\to N\to 0
  \]
  and a lift $B\colon L\to\ZZ^r$ of $\beta$. Then $D(G_\beta^1)$ is the cokernel of $(B^*,Q^*)\colon (\ZZ^r)^*\to L^*\oplus (\ZZ^s)^*$, so the sequence below on the left is exact.
  \[\xymatrix@R-2.5pc{
     & L^*\ar[dr]^-\pi\\
    (\ZZ^r)^*\ar[ur]^-{B^*}\ar[dr]_-{Q^*} & \oplus & D(G_\beta^1)\ar[r] & 0\\
     & (\ZZ^s)^* \ar[ur]
  }\qquad\qquad
  \raisebox{2.5pc}{\xymatrix{
    & & L^*\\
    0\ar[r] & N^*\ar[r]\ar@{^(->}[ur]^{\beta^*} & (\ZZ^r)^*\ar[r]_{Q^*} \ar[u]_{B^*} & (\ZZ^s)^*
  }}\]
  If $\mu\in\ker\pi$, then there exists $\lambda\in (\ZZ^r)^*$ such that $\mu=B^*(\lambda)$ and $0=Q^*(\lambda)$. As the bottom row of the right-hand diagram is exact, we have that $\lambda\in N^*$ and $\mu=\beta^*(\lambda)$. Note that $\lambda$ is unique as $\beta^*$ is injective. If $\mu$ is non-negative on $\sigma$, then by hypothesis, $\lambda=\phi^*(\lambda')$ for a unique $\lambda'\in N'^*$.  To conclude the proof, note that since $\Phi(\sigma)=\sigma'$, an element $\mu'\in N'^*$ is non-negative on $\sigma'$ if and only if $\Phi^*(\mu')$ is non-negative on $\sigma$.
\end{proof}

\begin{lemma}[``Isomorphism on tori implies GMS'']\label{lem:isom-on-tori-is-gms}
 Suppose $\phi$ is an isomorphism and $\cok\beta$ is finite. Then $\X_{\sigma,\beta}\to \X_{\sigma',\beta'}$ is a good moduli space morphism.
\end{lemma}
\begin{proof}
  First we consider the case when $\cok\Phi$ is finite. By Lemma \ref{lem:triangle=>ses-of-G_betas}, $G_\beta$ is an extension of $G_{\beta'}$ by $G_\Phi$. The induced map is then $[X_\sigma/G_\beta]=\bigl[[X_\sigma/G_\Phi]/G_{\beta'}\bigr]\to [X_{\sigma'}/G_{\beta'}]$. By Lemma \ref{lem:quot-by-kernel-is-gms2}, $[X_\sigma/G_\Phi]\to X_{\sigma'}$ is a good moduli space morphism. Since the property of being a good moduli space morphism can be checked locally in the smooth topology (even in the fpqc topology, \cite[Proposition 4.7]{Alper:good}), $[X_\sigma/G_\beta]\to [X_{\sigma'}/G_{\beta'}]$ is a good moduli space morphism.

  In general, we reduce to the case where $\cok\Phi$ is finite. Let $L''$ be the saturation of $\Phi(L)$ in $L'$, let $\beta''$ be the restriction of $\beta'$ to $L''$, and let $\sigma''$ be the cone $\sigma'$, regarded as a fan on $L''$. By assumption (i.e.~the case where $\cok\Phi$ is finite), the induced morphism $\X_{\sigma,\beta}\to \X_{\sigma'',\beta''}$ is a good moduli space morphism. Note that since $\cok\beta$ is finite, so is $\cok\beta''$, so by Lemma \ref{lem:adding-rank-to-L} the induced morphism $\X_{\sigma'',\beta''}\to \X_{\sigma',\beta'}$ is an isomorphism. Therefore the composition $\X_{\sigma,\beta}\to \X_{\sigma',\beta'}$ is a good moduli space morphism.
\end{proof}

\begin{lemma}[``Quotient of a GMS is a GMS'']\label{lem:quot-of-gms}
 Suppose $\beta$ (resp.~$\beta'$) factors as $L\xrightarrow{\beta_1} N_0\xrightarrow{\beta_0} N$ (resp.~$L'\xrightarrow{\beta_1'} N_0'\xrightarrow{\beta_0'} N'$). Suppose $\phi_0\colon N_0\to N_0'$ makes the following diagram commute:
 \[\xymatrix{
  L \ar[r]^\Phi\ar[d]_{\beta_0} \POS p+(0,.7) *+{\sigma}="s" & L'\ar[d]^{\beta_0'} \POS p+(0,.7) *+{\sigma'} \ar@{<<-} "s"\\
  N_0\ar[r]^{\phi_0}\ar[d]_{\beta_1} & N_0' \ar[d]^{\beta_1'}\\
  N\ar[r]^\phi & N'
 }\]
 Suppose
 \begin{enumerate}
  \item $\cok\beta_0$ and $\cok\beta_0'$ are finite,
  \item $\ker\beta_1$ and $\ker\beta_1'$ are free,
  \item\label{item:quotient-of-gms-kernels} $\phi_0$ induces an isomorphism between $\ker\beta_1$ and $\ker\beta_1'$, and
  \item\label{item:quotient-of-gms-cokernels} $\phi$ induces an isomorphism between $\cok\beta_1$ and $\cok\beta_1'$.
 \end{enumerate}
 Then $\X_{\sigma,\beta}\to \X_{\sigma',\beta'}$ is a good moduli space morphism if and only if $\X_{\sigma,\beta_0}\to \X_{\sigma',\beta_0'}$ is as well.
\end{lemma}
\begin{proof}
 The first two conditions, together with Lemma \ref{lem:triangle=>ses-of-G_betas} imply that the rows in the following diagram are exact:
 \[\xymatrix{
  0\ar[r] & G_{\beta_0}\ar[r]\ar[d] & G_\beta\ar[r]\ar[d] & G_{\beta_1}\ar[r]\ar[d]^\wr & 0\\
  0\ar[r] & G_{\beta_0'}\ar[r] & G_{\beta'}\ar[r] & G_{\beta_1'}\ar[r] & 0
 }\]
 The last two conditions imply that the induced morphism $C(\beta_1)\to C(\beta_1')$ is a quasi-isomorphism, which implies that the rightmost vertical map is an isomorphism between $G_{\beta_1}$ and $G_{\beta_1'}$. Therefore, the morphism $\X_{\sigma,\beta} \cong [\X_{\sigma,\beta_0}/G_{\beta_1}]\to [\X_{\sigma',\beta_0'}/G_{\beta_1'}]\cong \X_{\sigma',\beta'}$ is the quotient of the morphism $\X_{\sigma,\beta_0}\to \X_{\sigma',\beta_0'}$ by $G_{\beta_1}\cong G_{\beta_1'}$. Since the property of being a good moduli space morphism can be checked locally on the base in the smooth topology \cite[Proposition 4.7]{Alper:good}, $\X_{\sigma,\beta}\to\X_{\sigma',\beta'}$ is a good moduli space morphism if and only if $\X_{\sigma,\beta_0}\to\X_{\sigma',\beta_0'}$ is as well.
\end{proof}


\begin{lemma}[``Removing trivial generic stackiness is a GMS'']\label{lem:removing-trivial-stackiness-gms}
 Suppose $N=N'\oplus N_0$ and that $\beta$ factors through $N'$. Let $L'=L$ and $\sigma'=\sigma$. Then $[X_\sigma/G_{\beta}]\to [X_{\sigma'}/G_{\beta'}]$ is a good moduli space morphism.
\end{lemma}
\begin{proof}
 We have that $\beta = \beta'\oplus 0\colon L\oplus 0\to N'\oplus N_0$, so $G_\beta = G_{\beta'}\oplus G_0$, where $G_0$ acts trivially on $X_\sigma$. The map $[X_\sigma/G_0]\to X_\sigma$ is then a good moduli space morphism, so $[X_\sigma/G_\beta] = \bigl[[X_\sigma/G_0]/G_{\beta'}\bigr]\to [X_\sigma/G_{\beta'}]$ is as well.
\end{proof}

\begin{proposition}\label{prop:crush-unstable-is-gms}
 Let $\tau$ be an unstable face of $\sigma$, and let $\tau^\gp$ denote $(L\cap \tau)^\gp$. Suppose $L'=L/\tau^\gp$ and $N'=N/\beta(\tau^\gp)$. Then $(\Phi,\phi)$ induces a good moduli space morphism:
 \[\xymatrix{
 L \ar[r]^-\Phi\ar[d]_\beta \POS p+(0,.7) *+{\sigma}="s" & L/\tau^\gp\ar[d] \POS p+(0,.7) *+{\sigma'\rlap{$\;=\sigma/\tau$}} \ar@{<<-} "s"\\
 N\ar[r]^-\phi & N/\beta(\tau^\gp)
 }\]
\end{proposition}
\begin{proof}
 Consider the following diagram, in which the rows are exact. We define $K=\ker(\tau^\gp\to \beta(\tau^\gp))$.
 \[\xymatrix{
  K\ari[d] \ar@{=}[r] & K\ari[d]\\
  \tau^\gp\ari[r]\ar@{->>}[d] & L \ar@{->>}[r]\ar@{->>}[d]^{\beta_0} \POS p+(.7,.7) *+{\sigma}="s" & L/\tau^\gp\ar@{=}[d] \POS p+(.7,.7) *+{\sigma/\tau} \ar@{<<-} "s"\\
  \beta(\tau^\gp)\ar@{=}[d]\ari[r] & L/K\ar[d]^{\beta_1}\ar@{->>}[r] & L/\tau^\gp\ar[d]^{\beta_1'}\\
  \beta(\tau^\gp)\ari[r] & N\ar@{->>}[r]& N/\beta(\tau^\gp)
 }\]
 The top right square induces a good moduli space morphism by Lemma \ref{lem:quot-by-kernel-is-gms2}.
 Indeed, if $\lambda\in (L/K)^*$ is non-negative on $\sigma$, then it is 0 on $\tau$, as $\tau$ is unstable.  Hence $\lambda$ is induced from a unique element of $(L/\tau^{gp})^*$.

 Note that $\ker\beta_1 = (\ker\beta)/K$ is free since $K$ is saturated in $\ker\beta$ by construction. Conditions (\ref{item:quotient-of-gms-kernels}) and (\ref{item:quotient-of-gms-cokernels}) of Lemma \ref{lem:quot-of-gms} follow from the 5-lemma and the vertical equality in the left column, so it applies to complete the proof.
\end{proof}

\begin{lemma}[``Removing finite generic stackiness is a GMS'']\label{lem:removing-finite-generic-stackiness-gms}
 Suppose $\ker\phi=N_0$ is finite and $\phi$ is surjective. Let $L'=L$. Then $[X_\sigma/G_{\beta}]\to [X_{\sigma'}/G_{\beta'}]$ is a good moduli space morphism.
\end{lemma}
\begin{proof}
 Consider the left-hand diagram, in which the rows are exact.
 \[\xymatrix{
  0\ar[r] & 0\ar[r]\ar[d]_{\beta_0} & L\ar@{=}[r]\ar[d]_\beta & L\ar[r]\ar[d]^{\beta'} & 0\\
  0\ar[r] & N_0\ar[r] & N\ar[r]^\phi & N'\ar[r] & 0
 }\qquad
 \xymatrix{
  0\ar[r] & G_{\beta_0}\ar[r]\ar[d] & G_\beta\ar[r]\ar[d] & G_{\beta'}\ar[r]\ar[d] & 0\\
  0\ar[r] & 0\ar[r] & T_L\ar@{=}[r] & T_L\ar[r] & 0
 }\]
 We have that $\cok\beta_0=N_0$ is finite, so by Lemma \ref{lem:ses=>ses-of-G_betas} we get the induced right-hand diagram in which the rows are exact. In particular, the action of $G_{\beta_0}$ on $X_\sigma$ is trivial, so the map $[X_\sigma/G_{\beta_0}]\to X_\sigma$ is a good moduli space morphism. Thus $[X_\sigma/G_\beta]=\bigl[[X_\sigma/G_{\beta_0}]/G_{\beta'}\bigr]\to [X_\sigma/G_{\beta'}]$ is a good moduli space morphism.
\end{proof}

\begin{lemma}\label{lem:gms-of-BGs}
 If $f\colon G\to H$ is a morphism of algebraic groups, then the induced morphism $BG\to BH$ is a good moduli space morphism if and only if $f$ is surjective and $K=\ker(f)$ is linearly reductive.
\end{lemma}
\begin{proof}
 We have that $\spec k\to BH$ is a smooth cover, so by \cite[Proposition 4.7]{Alper:good}, $BG\to BH$ is a good moduli space morphism if and only if $BG\times_{BH}\spec k = [H/G]\to \spec k$ is as well. Since $[H/G]\cong BK\times H/\im(f)$, the structure map to $\spec k$ is a good moduli space morphism if and only if $H=\im(f)$ and $K$ is linearly reductive.
\end{proof}

\begin{lemma}\label{lem:gms-when-L=0}
 Suppose $L=L'=0$ and that $N$ and $N'$ are finite. Then $\X_{(\Phi,\phi)}$ is a good moduli space morphism if and only if $\phi\colon N\to N'$ is surjective.
\end{lemma}
\begin{proof}
 For $\beta\colon 0\to N$ with $N$ finite, we have that $\X_{0,\beta}$ is naturally $B(D(\ext^1(N,\ZZ)))$. By Lemma \ref{lem:gms-of-BGs}, we have that $\X_{(\Phi,\phi)}$ is a good moduli space morphism if and only if $\phi$ induces a surjection $D(\ext^1(N,\ZZ))\to D(\ext^1(N',\ZZ))$, which occurs if and only if it induces an injection $\ext^1(N',\ZZ)\to \ext^1(N,\ZZ)$.  We claim that this is equivalent to surjectivity of $\phi$.

 Indeed, for a short exact sequence of finite abelian groups
 \[
  0\to A\to B\to C\to 0
 \]
 we get a short exact sequence
 \[
  0\to \ext^1(C,\ZZ)\to \ext^1(B,\ZZ)\to \ext^1(A,\ZZ)\to 0.
 \]
 So if $\phi$ is surjective, the induced map $\ext^1(N',\ZZ)\to \ext^1(N,\ZZ)$ is injective.  Conversely, if $\phi$ is not surjective, then it has a non-trivial cokernel $C$. We then have that the induced map $\ext^1(N',\ZZ)\to \ext^1(N,\ZZ)$ factors as
 \[
  \ext^1(N',\ZZ)\to \ext^1(\im\phi,\ZZ)\to \ext^1(N,\ZZ)
 \]
 and the kernel of the first map is $\ext^1(C,\ZZ)$, which is (non-canonically) isomorphic to $C$, so $\ext^1(N',\ZZ)\to \ext^1(N,\ZZ)$ is not injective.
\end{proof}

\begin{proposition}
\label{prop:main-gms}
 Let $\tau$ be the pre-image cone $\Phi^{-1}(0)$. If $\cok\beta$ is finite, then $(\Phi,\phi)$ induces a good moduli space morphism if and only if
 \begin{enumerate}
 \item $\tau$ is unstable,
  \item $\phi$ is surjective, and
  \item $(\ker\phi)/\beta(\tau^\gp)$ is finite.
 \end{enumerate}
\end{proposition}
\begin{proof}
 Assume first that conditions (1)--(3) hold.  To show that $\X_{(\Phi,\phi)}$ is a good moduli space morphism, we factor $(\Phi,\phi)$ as follows, and show that each square induces a good moduli space morphism:
 \[\xymatrix{
  L \ar[d]_\beta\ar[r] \POS p+(0,.7) *+{\sigma}="s" & L/\tau^\gp \ar[d]\ar@{=}[r] \POS p+(0,.7) *+{\sigma/\tau}="st"
  \ar@{<<-}  "s"
  & L/\tau^\gp \ar[d]\ar[r] \POS p+(0,.7) *+{\sigma/\tau}="stt"
  \ar@{<-}_\sim "st"
  & L'\ar[d]^{\beta'} \POS p+(0,.7) *+{\sigma'}
  \ar@{<<-} "stt"\\
  N\ar[r] & N/\beta(\tau^\gp)\ar[r] & N'\ar@{=}[r] & N'
 }\]
 By condition (1) and Proposition \ref{prop:crush-unstable-is-gms}, the left square induces a good moduli space morphism.  By condition (3) and Lemma \ref{lem:removing-finite-generic-stackiness-gms}, the middle square induces a good moduli space morphism.  Lastly, condition (2) and the fact that $\cok\beta$ is finite shows that $\cok(\phi\circ\beta)$ is finite. Hence $\cok\beta'$ is finite, so Lemma \ref{lem:isom-on-tori-is-gms} shows that the right square induces a good moduli space morphism.

 Conversely, suppose $\X_{(\Phi,\phi)}$ is a good moduli space morphism. Base changing $\X_{(\Phi,\phi)}$ by the stacky torus of $\X_{\sigma',\beta'}$, we see that the following map of stacky fans also induces a good moduli space morphism by \cite[Proposition 4.7(i)]{Alper:good}:
 \[\xymatrix{
  L \ar[r]^\Phi\ar[d]_\beta \POS p+(0,.7) *+{\tau}="s" & L'\ar[d]^{\beta'} \POS p+(0,.7) *+{0} \ar@{<<-} "s"\\
  N\ar[r]^\phi & N'
 }\]
 By Lemma \ref{lem:quot-by-kernel-is-gms2}, the right square in the following diagram induces a good moduli space morphism, and so the composite map induces a good moduli space morphism:
 \[\xymatrix{
  L \ar[r]^\Phi\ar[d]_\beta \POS p+(0,.7) *+{\tau}="s" & L'\ar[d]^{\beta'}\ar[r]^{\beta'} \POS p+(0,.7) *+{0}="st" \ar@{<<-} "s" & N'\ar@{=}[d] \POS p+(0,.7) *+{0}\ar@{<-} "st" \\
  N\ar[r]^\phi & N'\ar@{=}[r] & N'
 }\]
Another application of Lemma \ref{lem:quot-by-kernel-is-gms2} (to the composite) shows that every linear functional on $N$ which is non-negative on $\tau$ must be induced from $N'$, and is therefore 0 on $\tau$.  This shows that $\tau$ is unstable.

Consider the following factorization of $(\Phi,\phi)$:
 \[\xymatrix{
  L \ar[d]_\beta\ar[r] \POS p+(0,.7) *+{\tau}="s" & L/\tau^\gp \ar[d]\ar[r] \POS p+(0,.7) *+{0}="st" \ar@{<<-} "s"  & L'\ar[d]^{\beta'} \POS p+(0,.7) *+{0} \ar@{<-} "st"\\
  N\ar[r] & N/\beta(\tau^\gp)\ar[r] & N'
 }\]
Proposition \ref{prop:crush-unstable-is-gms} shows that the square on the left induces a good moduli space morphism.  Then by Lemma \ref{l:gms-comp}, the square on the right also induces a good moduli space morphism. Thus we may assume $\sigma$ and $\sigma'$ are the zero cones. In this case, we must show that $\phi$ is surjective with finite kernel.

 Let $N=N_\tor\oplus N_f$, where $N_\tor$ is torsion and $N_f$ is free. Let $N' = N_\tor'\oplus N_f'\oplus N_0'$, where $N_0'$ is a direct complement to $\sat_{N'}(\beta'(L'))$, $N_\tor'$ is the torsion subgroup of $N'$, and $N_f'$ a direct complement to $N_\tor$ in $\sat_{N'}(\beta'(L'))$.
 Since $\X_{(\Phi,\phi)}$ is a good moduli space morphism, it must induce an isomorphism of the good moduli spaces of $\X_{0,\beta}$ and $\X_{0,\beta'}$. By Lemma \ref{lem:quot-by-kernel-is-gms2}, the good moduli space of $\X_{0,\beta}$ (resp.~$\X_{0,\beta'}=\X_{0,L'\to N_\tor'\oplus N_f'}\times B(D(N_0'))$) is Cartier dual to $N/N_\tor \cong N_f$ (resp.~$N_f'$). It follows that we may choose the complement $N_f'$ so that $\phi$ restricts to an isomorphism $\phi_f\colon N_f\to N_f'$. Let $\phi_0\colon N_\tor\to N_\tor'\oplus N_0'$ be the restriction of $\phi$ to $N_\tor$. It is clear that $\phi_0$ has finite kernel, so it remains to show that $\phi_0$ is surjective (in particular $N_0'=0$).

 Define $\beta_\tor\colon L\to N_\tor$, $\beta_f\colon L\to N_f$, $\beta_\tor'\colon L'\to N_\tor'$, $\beta_f'\colon L'\to N_f'$, and $\beta_0\colon 0\to N_0'$ so that $\beta=(\beta_\tor, \beta_f)$ and $\beta'=(\beta_\tor,\beta_f', \beta_0)$.
 By Proposition \ref{prop:isos-of-affines}, the following squares induce isomorphisms of stacky tori:
 \[\xymatrix{
  L\ar[r]^-\Delta \ar[d]_\beta & L\oplus L\ar[d]^-{(\beta_\tor,\beta_f)}\ar[r]^{0\oplus \id} & 0\oplus L\ar[d]^{0\oplus \beta_f}\\
  N\ar@{=}[r] & N_\tor\oplus N_f\ar@{=}[r] & N_\tor\oplus N_f
 }\]
 Hence $\X_{0,\beta}\cong \X_{0,0\to N_\tor}\times \X_{\beta_f}$, and similarly $\X_{0,\beta'}\cong \X_{0,0\to N_\tor'}\times \X_{0,\beta_f'}\times \X_{0,0\to N_0'}$.  Moreover, we see that $\X_{(\Phi,\phi)}$ is the product of morphisms $\X_{0,\beta_f}\to\X_{0,\beta'_f}$, $\X_{0,0\to N_\tor}\to \X_{0,0\to N'_\tor}$, and $\X_{0,0\to 0}\to \X_{0,0\to N'_0}$. By Lemma \ref{lem:product-of-gms-is-gms}, each of these morphisms must be a good moduli space morphism. Since the second morphism is a good moduli space morphism, Lemma \ref{lem:gms-when-L=0} shows that $\phi|_{N_\tor}\colon N_\tor\to N_\tor'$ is surjective. Since $\X_{0,0\to 0}$ is a non-stacky point and the third morphism is a good moduli space morphism, $\X_{0,0\to N_0'}$ must be a non-stacky point, so $N_0'=0$. It follows that $\phi_0$ is surjective.
\end{proof}

We now turn to Theorem \ref{thm:main-gms}.
\begin{proof}[Proof of Theorem \ref{thm:main-gms}]
 For $\sigma'\in \Sigma'$, let $X_{\sigma'}$ be the open subscheme of $X_{\Sigma'}$ corresponding to $\sigma'$. The property of being a good moduli space morphism can be checked Zariski locally on the base, so it is equivalent to checking that $[X_{\Phi^{-1}(\sigma')}/G_\beta]\to [X_{\sigma'}/G_{\beta'}]$ is a good moduli space morphism for each $\sigma'\in \Sigma'$. Therefore, by Proposition \ref{prop:main-gms}, the four conditions in Theorem \ref{thm:main-gms} imply that $\X_{\Sigma,\beta}\to \X_{\Sigma',\beta'}$ is a good moduli space morphism.

 Conversely, if $[X_{\Phi^{-1}(\sigma')}/G_\beta]\to [X_{\sigma'}/G_{\beta'}]$ is a good moduli space morphism, then $[X_{\Phi^{-1}(\sigma')}/G_\beta]\to \spec k$ is cohomologically affine since $[X_{\Phi^{-1}(\sigma')}/G_\beta]\to [X_{\sigma'}/G_{\beta'}]$ and $[X_{\sigma'}/G_{\beta'}]\to \spec k$ are cohomologically affine. So $\Phi^{-1}(\sigma')$ is a single cone $\sigma\in \Sigma$ (see Remark \ref{rmk:coh-affine-definitions-agree}). Good moduli space morphisms are surjective \cite[Theorem 4.16(i)]{Alper:good}, so by Lemma \ref{lem:coh-affine-surjection}, $\Phi(\sigma)=\sigma'$.  Proposition \ref{prop:main-gms} then shows that conditions (2)--(4) hold.
\end{proof}

\subsection{Examples}\label{subsec:toric-gms-examples}

\begin{example}
  Let $\Sigma = \raisebox{-2pt}{\begin{tikzpicture}[scale=.5,thick]
   \draw[<->] (1,0) -- (0,0) -- (0,1);
 \end{tikzpicture}}$ be the fan of $\AA^2\setminus 0$, and let $\beta\colon \ZZ^2\to 0$ be the zero map. Then $\X_{\Sigma,\beta}=[(\AA^2\setminus 0)/_{\smat{1&0\\ 0&1}}\GG_m^2]$. We see that both maximal cones of $\Sigma$ are unstable, so $(\Sigma,\beta)$ fails condition (i) of Corollary \ref{cor:main-gms-tv}. Therefore this stack does not have a toric variety good moduli space. Note that $\X_{\Sigma,\beta}$ is isomorphic to $[\PP^1/\GG_m]$, the prototypical example of a stack without a good moduli space.
\end{example}

\begin{example}\label{eg:gms-non-separated-line}
  Let $\Sigma = \raisebox{-2pt}{\begin{tikzpicture}[scale=.5,thick]
   \draw[<->] (1,0) -- (0,0) -- (0,1);
 \end{tikzpicture}}$ be the fan of $\AA^2\setminus 0$, and let $\beta=\begin{pmatrix}1&1\end{pmatrix}\colon \ZZ^2\to \ZZ$. Then $\X_{\Sigma,\beta}=[(\AA^2\setminus 0)/_{\smat{1&-1}}\GG_m]$. Condition (i) of Corollary \ref{cor:main-gms-tv} is satisfied, as the zero cone is the only unstable cone. However, the $\Sigma'$ of Notation \ref{not:fan-of-gms} consists of only the zero cone, so condition (ii) is not satisfied. Note that this $\X_{\Sigma,\beta}$ is the non-separated line (cf.~Warning \ref{warn:variety-gms}).
\end{example}

\begin{center}
\begin{tabular}{@{\extracolsep{1em}}cccc}
 \xymatrix{
   \ZZ^2\ar[d]_{\smat{1&\ng}}\ar[r] \POS p+(0,1.2) *+{
   \begin{tikzpicture}
      \clip (-.1,-.1) rectangle (1,1);
     \filldraw[fill=lightgray]
      (0,0) -- (9,0) -- (9,9) -- (0,9) -- cycle;
    \draw[thick,<->] (0,1) -- (0,0) -- (1,0);
   \end{tikzpicture}} & 0\ar[d]^{\phantom{\smat{1&\ng}}} \POS p+(0,.7) *+{\tikz \draw plot[mark=*] (0,0);}\\
   \ZZ \ar[r] & 0
 } &
 \xymatrix{
   \ZZ\ar[d]_2_(.4){\phantom{\id}}\ar[r]^2 \POS p+(0,.7) *+{\tikz \draw[->] (0,0) -- (.5,0);} & \ZZ\ar[d]^{\id} \POS p+(0,.7) *+{\tikz \draw[->] (0,0) -- (.5,0);}\\
   \ZZ \ar[r]^{\id} & \ZZ
 } &
 \xymatrix@C+2pc{
   \ZZ^2\ar[r]^{\smat{1&1\\ 0&2}} \ar[d]_{\smat{1&1\\ 0&2}} \POS p+(0,1) *+{
    \begin{tikzpicture}[scale=.4]
       \clip (-.5,-.5) rectangle (1.5,2.5);
       \foreach \x/\y/\z/\w in {9/0/0/8}
         \filldraw [draw=black,fill=lightgray] (\x,\y) -- (0,0) -- (\z,\w);
       \draw[help lines] (-10,-10) grid (10,10);
       \foreach \x/\y/\z/\w in {9/0/0/8}
         \draw [thick] (\x,\y) -- (0,0) -- (\z,\w);
   \end{tikzpicture}}   &
   \ZZ^2\ar@{=}[d]^{\phantom{\smat{1&1\\ 0&2}}}   \POS p+(0,1) *+{
    \begin{tikzpicture}[scale=.4]
       \clip (-.5,-.5) rectangle (1.5,2.5);
       \foreach \x/\y/\z/\w in {9/0/4/8}
         \filldraw [draw=black,fill=lightgray] (\x,\y) -- (0,0) -- (\z,\w);
       \draw[help lines] (-10,-10) grid (10,10);
       \foreach \x/\y/\z/\w in {9/0/4/8}
         \draw [thick] (\x,\y) -- (0,0) -- (\z,\w);
   \end{tikzpicture}} \\
   \ZZ^2\ar@{=}[r] & \ZZ^2
 } &
 \xymatrix{
  \ZZ^n \ar[d]_\beta \ar[r]^\beta \POS p+(0,.7) *+{\hhat\Sigma}="s" & N\ar@{=}[d] \POS p+(0,.7) *+{\Sigma} \ar@{<-} "s"\\
  N\ar@{=}[r] & N
 }\\
 Example \ref{eg:gms-A^2/G_m}&
 Example \ref{eg:gms-A^1/mu_2}&
 Example \ref{eg:gms-cox-construction}&
 Example \ref{ex:fantastack-gms}
\end{tabular}
\end{center}

\begin{example}\label{eg:gms-A^2/G_m}
 Theorem \ref{thm:main-gms} shows that this is a good moduli space morphism. Note that the unstable cone $\tau$ (the 2-dimensional cone) corresponds to the origin in $\AA^2$.
\end{example}

\begin{example}\label{eg:gms-A^1/mu_2}
 By Theorem \ref{thm:main-gms}, this map of stacky fans induces the good moduli space morphism $[\AA^1/\mu_2]\to \AA^1/\mu_2\cong \AA^1$.
\end{example}

\begin{example}\label{eg:gms-cox-construction}
 Theorem \ref{thm:main-gms} shows that the Cox construction of a toric variety is a good moduli space, and more generally that canonical stack morphisms are good moduli space morphisms. For example, applying Theorem \ref{thm:main-gms} to this map of stacky fans tells us that the toric variety on the right, the $A_1$ singularity, is the good moduli space of $[\AA^2/\mu_2]$ (cf.~Example \ref{eg:fantastack-A_1}).
\end{example}

\begin{example} \label{ex:fantastack-gms}
 As a special case of Theorem \ref{thm:main-gms}, we recover \cite[Theorem 5.5]{can}, which states that the fantastack $\F_{\Sigma,\beta}$ has the toric variety $X_\Sigma$ as its good moduli space.  Given a fan $\Sigma$ on a lattice $N$, let $\beta\colon \ZZ^n\to N$ and $\hhat\Sigma$ be as in Definition \ref{def:fantastack}. Then Theorem \ref{thm:main-gms} shows that the displayed map of stacky fans induces a good moduli space morphism, as desired.
\end{example}

It is often useful to think about a toric stack as ``sandwiched'' between its canonical stack and its good moduli space (if it has one). In this way, we often regard a toric stack as a ``partial good moduli space'' of its canonical stack or as a ``partial stacky resolution'' of its good moduli space.
\begin{example}
 Consider the stacky fan $(\Sigma,\beta)$ shown in the center below. On the left we have the stacky fan of the corresponding canonical stack (see Section \ref{sec:canonical-stacks}). On the right we have a toric variety. By Theorem \ref{thm:main-gms}, the two morphisms of stacky fans induce good moduli space morphisms of toric stacks.
 
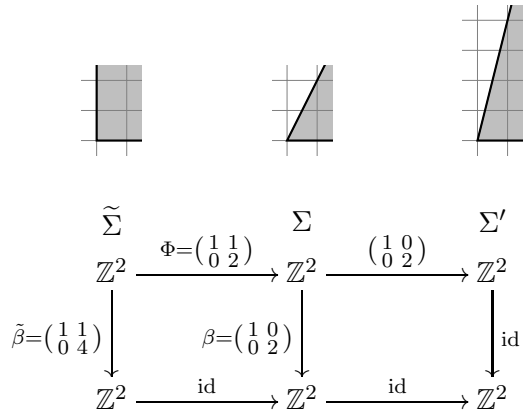
\[\xymatrix@!0 @R+2pc @C+4pc{
   \begin{tikzpicture}[scale=.4]
       \clip (-.5,-.5) rectangle (1.5,2.5);
       \foreach \x/\y/\z/\w in {9/0/0/8}
         \filldraw [draw=black,fill=lightgray] (\x,\y) -- (0,0) -- (\z,\w);
       \draw[help lines] (-10,-10) grid (10,10);
       \foreach \x/\y/\z/\w in {9/0/0/8}
         \draw [thick] (\x,\y) -- (0,0) -- (\z,\w);
   \end{tikzpicture} \POS p+(0,-1) *+{\ttilde\Sigma} &
   \begin{tikzpicture}[scale=.4]
       \clip (-.5,-.5) rectangle (1.5,2.5);
       \foreach \x/\y/\z/\w in {9/0/4/8}
         \filldraw [draw=black,fill=lightgray] (\x,\y) -- (0,0) -- (\z,\w);
       \draw[help lines] (-10,-10) grid (10,10);
       \foreach \x/\y/\z/\w in {9/0/4/8}
         \draw [thick] (\x,\y) -- (0,0) -- (\z,\w);
   \end{tikzpicture} \POS p+(0,-1) *+{\Sigma} &
   \begin{tikzpicture}[scale=.4]
       \clip (-.5,-.5) rectangle (1.5,4.5);
       \foreach \x/\y/\z/\w in {9/0/2/8}
         \filldraw [draw=black,fill=lightgray] (\x,\y) -- (0,0) -- (\z,\w);
       \draw[help lines] (-10,-10) grid (10,10);
       \foreach \x/\y/\z/\w in {9/0/2/8}
         \draw [thick] (\x,\y) -- (0,0) -- (\z,\w);
   \end{tikzpicture} \POS p+(0,-1) *+{\Sigma'}
   \\
   \ZZ^2\ar[r]^{\Phi=\smat{1&1\\ 0&2}} \ar[d]_{\tilde\beta=\smat{1&1\\ 0&4}}& \ZZ^2\ar[d]_{\beta=\smat{1&0\\ 0&2}}\ar[r]^{\smat{1&0\\ 0&2}} & \ZZ^2\ar[d]^\id \\
   \ZZ^2\ar[r]^\id & \ZZ^2 \ar[r]^\id & \ZZ^2
 }\]
 We can easily see that $X_\Sigma$ is the $A_1$ singularity $\AA^2/_{\smat{1&1}}\mu_2$, and that $G_\beta$ is $\mu_2$, but it is easiest to describe the action of $\mu_2$ on $X_\Sigma$ in terms of the canonical stack.

 The canonical stack is $\X_{\ttilde\Sigma,\tilde\beta}=[\AA^2/_{\smat{1&\ng}}\mu_4]$. We get the induced short exact sequence (cf.~Lemma \ref{lem:triangle=>ses-of-G_betas})
 \[\xymatrix@R-1.5pc{
  0\ar[r] & G_\Phi \ar@{=}[d]\ar[r]& G_{\tilde\beta}\ar@{=}[d]\ar[r]& G_\beta\ar@{=}[d]\ar[r] & 0 \\
  0\ar[r] & \mu_2\ar[r]^2 & \mu_4\ar[r] & \mu_2\ar[r] & 0
 }\]
 We can therefore express $\X_{\Sigma,\beta}$ as $\bigl[(\AA^2/\mu_2)\bigm/(\mu_4/\mu_2)\bigr]$. We can view this either as a ``partial good moduli space'' of $\X_{\ttilde\Sigma,\tilde\beta}=[\AA^2/_{\smat{1&\ng}}\mu_4]$ or as a ``partial stacky resolution'' of the singular toric variety $\X_{\Sigma'}=\AA^2/_{\smat{1&\ng}}\mu_4$.
\end{example}

\begin{example}
 Here is another example of a non-smooth toric stack which is not a scheme. Consider the stacky fan $(\Sigma,\beta)$ shown in the center below. On the left we have the stacky fan of the corresponding canonical stack (see Section \ref{sec:canonical-stacks}). On the right we have a toric variety. By Theorem \ref{thm:main-gms}, the two morphisms of stacky fans induce good moduli space morphisms of toric stacks.
 
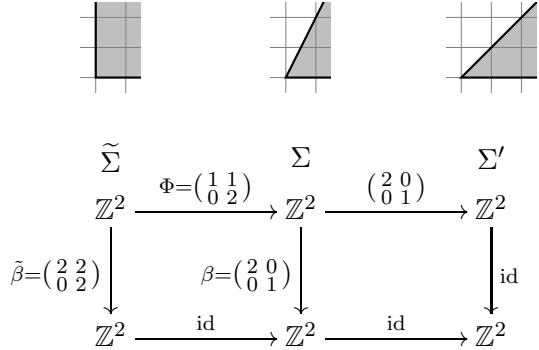
\[\xymatrix@!0 @R+2pc @C+4pc{
   \begin{tikzpicture}[scale=.4]
       \clip (-.5,-.5) rectangle (1.5,2.5);
       \foreach \x/\y/\z/\w in {9/0/0/8}
         \filldraw [draw=black,fill=lightgray] (\x,\y) -- (0,0) -- (\z,\w);
       \draw[help lines] (-10,-10) grid (10,10);
       \foreach \x/\y/\z/\w in {9/0/0/8}
         \draw [thick] (\x,\y) -- (0,0) -- (\z,\w);
   \end{tikzpicture} \POS p+(0,-1) *+{\ttilde\Sigma} &
   \begin{tikzpicture}[scale=.4]
       \clip (-.5,-.5) rectangle (1.5,2.5);
       \foreach \x/\y/\z/\w in {9/0/4/8}
         \filldraw [draw=black,fill=lightgray] (\x,\y) -- (0,0) -- (\z,\w);
       \draw[help lines] (-10,-10) grid (10,10);
       \foreach \x/\y/\z/\w in {9/0/4/8}
         \draw [thick] (\x,\y) -- (0,0) -- (\z,\w);
   \end{tikzpicture} \POS p+(0,-1) *+{\Sigma} &
   \begin{tikzpicture}[scale=.4]
       \clip (-.5,-.5) rectangle (2.5,2.5);
       \foreach \x/\y/\z/\w in {9/0/8/8}
         \filldraw [draw=black,fill=lightgray] (\x,\y) -- (0,0) -- (\z,\w);
       \draw[help lines] (-10,-10) grid (10,10);
       \foreach \x/\y/\z/\w in {9/0/8/8}
         \draw [thick] (\x,\y) -- (0,0) -- (\z,\w);
   \end{tikzpicture} \POS p+(0,-1) *+{\Sigma'} &
   \\
   \ZZ^2\ar[r]^{\Phi=\smat{1&1\\ 0&2}} \ar[d]_{\tilde\beta=\smat{2&2\\ 0&2}}& \ZZ^2\ar[d]_{\beta=\smat{2&0\\ 0&1}}\ar[r]^{\smat{2&0\\ 0&1}} & \ZZ^2\ar[d]^\id \\
   \ZZ^2\ar[r]^\id & \ZZ^2 \ar[r]^\id & \ZZ^2
 }\]
 Like the previous example, $\X_{\Sigma,\beta}$ is a quotient of the $A_1$ singularity $X_\Sigma=\AA^2/_{\smat{1&1}}\mu_2$ by an action of $\mu_2$. The canonical stack over it is $\X_{\ttilde\Sigma,\tilde\beta}=\bigl[\AA^2/_{\smat{1&0\\ 0&1}}(\mu_2\times\mu_2)\bigr]$, and it is the ``partial good moduli space'' $\bigl[(\AA^2/_{\smat{1&1}}\mu_2)/((\mu_2\times\mu_2)/_{\smat{1&1}}\mu_2)\bigr]$. The toric stack $\X_{\Sigma,\beta}$ has $\AA^2$ as its good moduli space.

 In Example \ref{eg:fantastack-A_1}, we constructed a stack which ``resolves the $A_1$ singularity by introducing stackiness''. In the same informal language, this example \emph{introduces a singularity} at a smooth point of $\AA^2$ by introducing stackiness.
\end{example}


\section{Moduli Interpretation of Smooth Toric Stacks}\label{sec:moduli}

A morphism $f\colon Y\to \PP^n$ is equivalent to the data of a line bundle $\L=f^*\O_{\PP^n}(1)$ and a choice of $n+1$ sections $\O_Y^{n+1}\to \L$ which generate $\L$. Cox generalized this moduli interpretation to smooth toric varieties \cite{cox:smooth}, and Perroni generalized it further to smooth toric Deligne-Mumford stacks \cite{perroni}. The main goal for this section is to generalize it further to smooth non-strict toric stacks.

In fact, we will see (Remark \ref{rmk:moduli-natural-class}) that smooth non-strict toric stacks are precisely the moduli stacks parametrizing tuples of generalized effective Cartier divisors (see \cite[Example 2.5a]{parabolic}) satisfying any given linear relations and any given intersection relations.

\begin{proposition}[{\cite[Expos\'e VIII, Proposition 4.1]{sga3}}]\label{prop:diagonalizable-torsor}
 Let $G$ be a diagonalizable group scheme, and $Y$ a scheme. Suppose $\A$ is a quasi-coherent $\O_Y$-algebra, together with an action of $G$ $($i.e.~a grading $\A=\bigoplus_{\chi\in D(G)}\A_\chi)$. Then $\Spec_Y \A$ is a $G$-torsor if and only if
 \begin{itemize}
  \item $\A_\chi$ is a line bundle for each $\chi\in D(G)$, and
  \item the homomorphism induced by multiplication $\A_\chi\otimes_{\O_Y}\A_{\chi'}\to \A_{\chi+\chi'}$ is an isomorphism.
 \end{itemize}
\end{proposition}
Since any $G$-torsor is affine over $Y$, it is clear that any $G$-torsor is of this form.

\begin{notation}
 Given a collection of line bundles $\L_1,\dots, \L_n\in \pic(Y)$ and $\a=(a_1,\dots, a_n)\in \ZZ^n$, let $\L^\a = \L_1^{\otimes a_1}\otimes\cdots \otimes \L_n^{\otimes a_n}$.
\end{notation}

Let $\beta\colon \ZZ^n\to N$ be the morphism of lattices so that $G_\beta=G\subseteq \GG_m^n$. Then we have the presentation $N^*\xrightarrow{\beta^*} \ZZ^n\xrightarrow{\phi} D(G)\to 0$. A quasi-coherent $\O_Y$-algebra as in Proposition \ref{prop:diagonalizable-torsor} is therefore equivalent to a collection of line bundles $\L_1,\dots, \L_n\in \pic(Y)$ with isomorphisms $c_\psi\colon \O_Y\xrightarrow\sim \L^{\beta^*(\psi)}$ for $\psi\in N^*$, such that $c_{\psi}\otimes c_{\psi'}=c_{\psi+\psi'}\colon \O_Y\to \L^{\beta^*(\psi)+\beta^*(\psi')}$.

\begin{definition}\label{def:(Sigma,beta)-collection}
 Suppose $\Sigma$ is a subfan of the fan of $\AA^n$, and $\beta\colon \ZZ^n\to N$ is a lattice homomorphism with finite cokernel. A \emph{$(\Sigma,\beta)$-collection} on a scheme $Y$ consists of
 \begin{itemize}
  \item an $n$-tuple of line bundles $(\L_1,\dots, \L_n)$,
  \item global sections $s_i\in H^0(Y,\L_i)$ so that for each point $y\in Y$, there is a cone $\sigma\in \Sigma$ so that $s_i(y)\neq 0$ for all $e_i\not\in \sigma$.
  \item trivializations $c_\psi\colon \O_Y\xrightarrow\sim \L^{\beta^*(\psi)}$ for each $\psi\in N^*$, satisfying the compatibility condition $c_{\psi}\otimes c_{\psi'}=c_{\psi+\psi'}$.
 \end{itemize}
 An \emph{isomorphism} of $(\Sigma,\beta)$-collections is an $n$-tuple of isomorphisms of line bundles respecting the associated sections and trivializations.
\end{definition}
\begin{remark}\label{rmk:suppress-trivializations}
 Note that since $N^*$ is a free subgroup of $\ZZ^n$, it suffices to specify $c_{\psi}$ where $\psi$ varies over a basis of $N^*$. If specified this way, the isomorphisms do not need to satisfy any compatibility condition. Different choices of these trivializations are related by the action of the torus (see Remark \ref{rmk:moduli-torus-action}), so we often suppress the trivializations.
\end{remark}
\begin{remark}\label{rmk:effective-pseudodivisors}
 A line bundle with section is a \emph{generalized effective Cartier divisor}. If the section is non-zero on each irreducible component of $Y$, the vanishing locus of the section is an effective Cartier divisor. Therefore, a $(\Sigma,\beta)$-collection can be defined as an $n$-tuple of generalized effective Cartier divisors $(D_1,\dots, D_n)$ such that $\sum a_iD_i$ is linearly equivalent to zero whenever $(a_1,\dots, a_n)\in N^*$ and such that $D_{i_1}\cap\dots\cap D_{i_k}=\varnothing$ whenever $e_{i_1},\dots, e_{i_k}$ do not all lie on a single cone of $\Sigma$. Here we are suppressing the trivializations as in Remark \ref{rmk:suppress-trivializations}.
\end{remark}

\begin{remark}
 Given a morphism $Y'\to Y$ and a $(\Sigma,\beta)$-collection on $Y$, we may pull back the line bundles, sections, and trivializations to produce a $(\Sigma,\beta)$-collection on $Y'$. This makes the category of $(\Sigma,\beta)$-collections into a fibered category over the category of schemes.
\end{remark}

\begin{theorem}[{Moduli interpretation of smooth $\X_{\Sigma,\beta}$}]\label{thm:moduli}
 Let $\Sigma$ be a subfan of the fan for $\AA^n$, and let $\beta\colon \ZZ^n\to N$ be a lattice homomorphism with finite cokernel. Then $\X_{\Sigma,\beta}$ represents the fibered category of $(\Sigma,\beta)$-collections.
\end{theorem}
\begin{proof}
 A morphism $f\colon Y\to [\AA^n/G_\beta]$ consists of a $G_\beta$-torsor $P\to Y$ and a $G_\beta$-equivariant morphism $P\to \AA^n$. By Proposition \ref{prop:diagonalizable-torsor}, the data of a $G_\beta$-torsor is equivalent to a $D(G_\beta)$-graded quasi-coherent sheaf of algebras $\A$ such that $\A_\chi$ is a line bundle for each $\chi\in D(G_\beta)$. A $G_\beta$-equivariant morphism $P\to \AA^n$ is then equivalent to a homomorphism of $\O_T$-algebras $\bigoplus_{\a\in \NN^n}\O_T\to \A$ which respects the $D(G_\beta)$-grading (the $D(G_\beta)$-grading on the former algebra is induced by the $\ZZ^n$-grading and the homomorphism $\phi\colon \ZZ^n\to D(G_\beta)$).
 This is equivalent to homomorphisms of $\O_T$-modules $s_i\colon \O_T\to \A_{\phi(e_i)}$. Under this correspondence, the vanishing locus of $s_i$ is the pre-image $[\AA^{n-1}_i/G_\beta]$, where $\AA^{n-1}_i$ is the $i$-th coordinate hyperplane. In particular, $(\A_{\phi(e_1)},\dots, \A_{\phi(e_n)},s_1,\dots, s_n)$, along with the implicit trivializations, form a $(\Sigma,\beta)$-collection if and only if $f$ factors through the open substack $\X_{\Sigma,\beta}$.

 It is straightforward to verify that the above correspondence induces an equivalence.
\end{proof}

\begin{remark}\label{rmk:moduli-torus-action}
 Carefully following the construction in the proof shows that the action of the torus is as follows. Suppose $N^*\subseteq \ZZ^n$ is the sublattice of trivialized line bundles. Then the trivializations $c_\psi$ have natural weights of the torus $T=\hom_\gp(N^*,\GG_m)$ associated to them. $T$ acts on the trivializations via these weights.
\end{remark}

\begin{remark}\label{rmk:moduli-open-immersions}
 This moduli interpretation is stable under base change by open immersions. Suppose a morphism $Y\to \X_{\Sigma,\beta}$ corresponds to the $(\Sigma,\beta)$-collection $(\L_i,s_i,c_\psi)$. Let $\Sigma'$ be a subfan of $\Sigma$. Then the pullback $Y'=Y\times_{\X_{\Sigma,\beta}}\X_{\Sigma',\beta}$ is the open subscheme of $Y$ where a subset of sections may simultaneously vanish only if $\Sigma'$ contains the cone spanned by the rays corresponding to those sections.
\end{remark}

\begin{remark}\label{rmk:moduli-natural-class}
 As a converse to Theorem \ref{thm:moduli}, note that any set of intersection relations among an $n$-tuple of generalized effective Cartier divisors (i.e.~any specification of which subsets of divisors should have empty intersection) determines a subfan $\Sigma$ of the fan of $\AA^n$. Furthermore, any\footnote{There is \emph{one} required relationship between the intersection relations and the trivializations. Namely, if the intersection of a \emph{single} Cartier divisor is required to be empty (i.e.~if the corresponding section of the line bundle is nowhere vanishing), then the line bundle must be trivialized.
 That is, if the intersection relations explicitly require the section to trivialize the line bundle, then it must be trivialized.} compatible collection of trivializations determines a subgroup $N^*=\{\a\in \ZZ^n|\L^\a$ is trivialized$\}$. The dual of the inclusion of $N^*$ is a lattice homomorphism $\beta\colon \ZZ^n\to N$ with finite cokernel. Then $\X_{\Sigma,\beta}$ is the moduli stack of $n$-tuples of generalized effective Cartier divisors with the given intersection relations and linear relations.
\end{remark}

\subsection{Examples}
The simplest examples to describe are fantastacks. See Notation \ref{not:fantastacks} and Examples \ref{eg:fantastack-double-origin}--\ref{eg:fantastack-square-cone} for an explanation of the notation used below.

\begin{remark}
 For any smooth stacky fan $(\Sigma, \beta\colon L\to N)$, one may modify $\Sigma$ by first removing all unstable cones which $\beta$ does not map to $0$ and then including cones so that $\Sigma$ is induced by a fan on $N$. This shows that any smooth toric stack contains an open substack which has a toric open immersion into a fantastack. Remark \ref{rmk:moduli-open-immersions} therefore allows us to understand the moduli interpretation of non-fantastack smooth toric stacks by appropriately modifying the intersection relations.
\end{remark}

\begin{remark}\label{rmk:moduli-linear-relations}
 We explicitly obtain linear relations by choosing a basis for $N^*$. For each basis element $\psi$, we get a trivialization of $\L^{\beta^*(\psi)}$. That is, we get trivializations of the divisors whose coefficients appear in the rows of $\beta^*$.
\end{remark}

\begin{center}
\begin{tabular}{@{\extracolsep{3em}}ccc}
 $\begin{tikzpicture}
 \clip (-1.5,-1.5) rectangle (1.5,1.5);
 \foreach \x/\y/\z/\w in {9/0/0/9, 0/9/-9/-9, -9/-9/9/0}
   \filldraw [draw=black,fill=lightgray] (\x,\y) -- (0,0) -- (\z,\w);
 \draw[help lines] (-10,-10) grid (10,10);
 \foreach \x/\y/\z/\w in {9/0/0/9, 0/9/-9/-9, -9/-9/9/0}
   \draw [thick] (\x,\y) -- (0,0) -- (\z,\w);
 \foreach \x/\y/\dottext in {1/0/1, 0/1/2, -1/-1/3}
   \filldraw[draw=black,fill=white] (\x,\y) node {$\dottext$} circle (6pt);
 \end{tikzpicture}$
 &
  $\begin{tikzpicture}
   \draw [thick,->] (0,0) -- (2,0);
 \filldraw[draw=black,fill=white] (1,0) node[above=4pt] {$2$} circle (6pt)
   (1,0) node {\small $1$} circle (4pt);
 \end{tikzpicture}$
 &
 $\begin{tikzpicture}
 \clip (-.5,-.5) rectangle (1.5,2.5);
 \foreach \x/\y/\z/\w in {9/0/4/8}
   \filldraw [draw=black,fill=lightgray] (\x,\y) -- (0,0) -- (\z,\w);
 \draw[help lines] (-10,-10) grid (10,10);
 \foreach \x/\y/\z/\w in {9/0/4/8}
   \draw [thick] (\x,\y) -- (0,0) -- (\z,\w);
 \foreach \x/\y/\dottext in {1/0/1, 1/2/2}
   \filldraw[draw=black,fill=white] (\x,\y) node {$\dottext$} circle (6pt);
 \end{tikzpicture}$
 \\
 Example \ref{eg:moduli-P^2}&
 Example \ref{eg:moduli-double-origin}&
 Example \ref{eg:moduli-A_1}
\end{tabular}
\end{center}

We follow the less formal approach to $(\Sigma,\beta)$-collections explained in Remark \ref{rmk:effective-pseudodivisors}.

\begin{example}\label{eg:moduli-P^2}
 A morphism to the leftmost stack is a choice of three generalized effective Cartier divisors $D_1, D_2, D_3$ such that $D_1\cap D_2\cap D_3=\varnothing$ (because no cone contains all three dots), and so that $D_1-D_3\sim \varnothing$ and $D_2-D_3\sim \varnothing$ (because $\beta^*=\smat{1&0&\ng\\ 0&1&\ng}$; see Remark \ref{rmk:moduli-linear-relations}). Here, $\varnothing$ denotes the empty divisor.

 In other words, it is a choice of a line bundle and three global sections that do not all vanish at any point. This is the usual description of morphisms to $\PP^2$.
\end{example}
\begin{example}\label{eg:moduli-double-origin}
  A morphism to the middle stack is a choice of two generalized effective Cartier divisors $D_1$ and $D_2$ so that $D_1+D_2\sim \varnothing$ (because $\beta^*=(1\ 1)$; see Remark \ref{rmk:moduli-linear-relations}). Notice two particular morphisms from $\AA^1$ to this stack; one given by setting $D_1=0$ and $D_2=\varnothing$, and another by setting $D_1=\varnothing$ and $D_2=0$. Indeed, the open substack where we impose the condition $D_1\cap D_2=\varnothing$ is the non-separated line.
\end{example}
\begin{example}\label{eg:moduli-A_1}
  A morphism to the rightmost stack $\X$ is a choice of two generalized effective Cartier divisors $D_1$ and $D_2$ so that $D_1+D_2\sim \varnothing$ and $2D_2\sim \varnothing$ (because $\beta^*=\smat{1&1\\ 0&2}$; see Remark \ref{rmk:moduli-linear-relations}). Since there is a single cone that contains all the dots, there is no intersection condition on $D_1$ and $D_2$. Notice that since two such divisors satisfy the condition $\D_1+\D_2\sim \varnothing$, we get a natural transformation $\hom(-,\X)\to \hom(-,\Y)$, where $\Y$ is the stack of Example \ref{eg:moduli-double-origin}. The corresponding map $\X\to\Y$ can be seen as the vertical projection in the displayed pictures.
\end{example}

\subsection{Smooth Non-strict Toric Stacks}

In this subsection we use the moduli interpretation to show that any smooth non-strict toric stack is an essentially trivial gerbe over a toric stack.
Suppose $Y\to \X_{\Sigma,\beta} = [X_\Sigma/G_\beta]$ is the morphism to a smooth toric stack corresponding to the $(\Sigma,\beta)$-collection $(\L_i,s_i,c_\psi)$. It factors through the closed substack corresponding to the $j$-th coordinate hyperplane of $\AA^n$ if and only if $s_j=0$. Theorem \ref{thm:moduli} (together with Remark \ref{rmk:smooth-gen-stacky<smooth}) therefore gives us the following moduli interpretation of smooth non-strict toric stacks.

\begin{corollary}[{of Theorem \ref{thm:moduli}}]\label{cor:moduli-gen-stacky}
 The smooth non-strict toric substack $\Z$ of $\X_{\Sigma,\beta}$ corresponding to a coordinate subspace $H$ of $\AA^n$ has the following moduli interpretation: morphisms $Y\to \Z$ correspond to $(\Sigma,\beta)$-collections on $Y$ in which we require that $s_j=0$ if $H$ does not contain the $j$-th coordinate axis.
\end{corollary}

\begin{definition}\label{def:root-stack}
 Suppose $\K$ is a line bundle on a stack $\X$ and $b$ is a positive integer. The \emph{root stack} $\sqrt[b]{\K/\X}$ is defined as the fiber product in the following diagram, where the map $\X\to B\GG_m$ is the one induced by $\K$:
 \[\xymatrix{
  \sqrt[b]{\K/\X} \ar[r]\ar[d] & B\GG_m\ar[d]^{\hat{}\,b}\\
  \X\ar[r] & B\GG_m
 }\]
 The map $\hat{}\,b\colon B\GG_m\to B\GG_m$ is given by sending a line bundle to its $b$-th tensor power. It is induced by the group homomorphism $\GG_m\to \GG_m$ given by $t\mapsto t^b$.

 If $\underline{\K}=(\K_1,\dots, \K_r)$ is an $r$-tuple of line bundles and $\b=(b_1,\dots, b_r)$ is an $r$-tuple of positive integers, we similarly define $\sqrt[\b]{\underline\K/\X}$ as $\X\times_{B\GG_m^r}B\GG_m^r$, where the map $\X\to B\GG_m^r$ is induced by the tuple $(\K_1,\dots, \K_r)$ and the map $\hat{}\,\b\colon B\GG_m^r\to B\GG_m^r$ is induced by the homomorphism $\GG_m^r\to \GG_m^r$ given by $(t_1,\dots, t_r)\mapsto (t_1^{b_1},\dots, t_r^{b_r})$. It is straightforward to check that $\sqrt[\b]{\underline\K/\X}$ is the fiber product of the $\sqrt[b_i]{\K_i/\X}$ over $\X$.
\end{definition}

\begin{remark}
 Explicitly, a morphism from a scheme (or stack) $Y$ to the root stack $\sqrt[\b]{\underline\K/\X}$ is a morphism $f\colon Y\to \X$, an $r$-tuple of line bundles $(\L_1,\dots, \L_r)$, and isomorphisms $\L_i^{\otimes b_i}\cong f^*\K_i$.
\end{remark}

\begin{definition}\label{def:essentially-trivial-gerbe}
 We say that $\X\to \Y$ is an \emph{essentially trivial gerbe} if $\X$ is of the form $B(\GG_m^s)\times \sqrt[\b]{\underline \K/\Y}$.
\end{definition}

\begin{proposition}\label{prop:essentially-trivial-gerbe}
 Let $\Z$ be a smooth non-strict toric stack. Then $\Z$ is an essentially trivial gerbe over a smooth toric stack.
\end{proposition}
\begin{proof}
 Suppose $\Z$ is a closed torus-invariant substack of a toric stack $\X_{\Sigma,\beta}$, with $\Sigma$ a subfan of the fan of $\AA^n$ (this is possible by Remark \ref{rmk:smooth-gen-stacky<smooth}). Let $\D_1,\dots, \D_n$ be the torus-invariant divisors of $\X_{\Sigma,\beta}$. Without loss of generality, $\Z=\D_1\cap\cdots\cap \D_l$. By Corollary \ref{cor:moduli-gen-stacky}, $\Z$ is the stack of $(\Sigma,\beta)$-collections where $s_i=0$ for $1\le i\le l$.

 Let $\Sigma'$ be the restriction of $\Sigma$ to the sublattice $\ZZ^{n-l}\subseteq \ZZ^n$ given by the last $n-l$ coordinates, let $N'^*=\ZZ^{n-l}\cap N^*$, and let $\beta'\colon \ZZ^{n-l}\to N'$ be the dual to the inclusion. For each $i$ between $1$ and $l$, let $b_i$ be the smallest positive integer (if it exists) so that $(0,\dots, 0,b_i,0,\dots, 0, a_{i,l+1},\dots, a_{i,n})\in N^*$.
 Without loss of generality, we may assume these integers exist for $1\le i\le r$. Define $\K_i=\L_{l+1}^{-a_{i,l+1}}\otimes\cdots \otimes \L_n^{-a_{i,n}}$.

 Suppose $(\L_i,s_i,c_\psi)$ is a $(\Sigma,\beta)$-collection on a scheme such that $s_i=0$ for $1\le i\le l$. Then the last $n-l$ line bundles with sections form a $(\Sigma',\beta')$-collection, the line bundles $\L_i$ for $r<i\le l$ satisfy no relations, and for $1\le i\le r$, we have isomorphisms $\L_i^{b_i}\cong \K_i$. Therefore, a morphism to $\Z$ is precisely the data of a morphism to $B(\GG_m^{l-r})\times \sqrt[(b_1,\dots, b_r)]{(\K_1,\dots, \K_r)/\X_{\Sigma',\beta'}}$.
\end{proof}

{\section*{Appendix A: Short Exact Sequences of \texorpdfstring{$G_\beta$s}{G\_betas}}
\renewcommand{\thesection}{A}
\refstepcounter{section}
\label{sec:ses-of-G_betas}

In this appendix, we prove some results which allow us to relate the groups of Definition \ref{def:G_beta} to one another. The basic advantage of expressing a group $G$ as an extension of a quotient $H$ by a normal subgroup $N$ is that any quotient stack $[X/G]$ can be identified with $\bigl[[X/N]/H\bigr]$. This observation is used heavily throughout this paper and \cite{toricartin2}.

We refer the reader to \cite{gelfand-manin} for the relevant homological algebra.

\begin{lemma}\label{lem:triangle=>ses-of-G_betas}
 Suppose $L$, $L'$, and $N'$ are finitely generated abelian groups. Suppose $\Phi\colon L\to L'$ and $\beta'\colon L'\to N'$ are homomorphisms. Suppose $\ker \Phi$ and $\ker \beta'$ are free, and $\cok\Phi$ is finite. Then we have the following diagram, in which the rows are exact and the morphisms to $D(L^{*})$ and $D(L'^{*})$ are the ones described immediately after Definition \ref{def:G_beta}:
 \[\xymatrix{
  0\ar[r] & G_\Phi\ar[r]\ar@{=}[d] & G_{\beta'\circ \Phi}\ar[r]\ar[d] & G_{\beta'}\ar[r]\ar[d] & 0\\
  0\ar[r] & G_\Phi\ar[r] & D(L^{*})\ar[r] & D(L'^{*})\ar[r] & 0
 }\]
\end{lemma}
\begin{proof}
 By the octahedron axiom, the commutative triangle $L\xrightarrow\Phi L'\xrightarrow{\beta'} N'$ induces an exact triangle on cones $C(\Phi)\to C(\beta'\circ \Phi)\to C(\beta')\to C(\Phi)[1]$. This induces an exact triangle of duals, which induces the long exact sequence of homology groups
 \[
  0\to D(G_{\beta'}^{0})\to D(G_{\beta'\circ \Phi}^{0})\to \underbrace{D(G_\Phi^{0})}_0\to D(G_{\beta'}^{1})\to D(G_{\beta'\circ \Phi}^{1})\to D(G_\Phi^{1})\to 0.
 \]
 Since $\cok\Phi$ is finite, $D(G_\Phi^{0})=0$. We therefore get the following diagram with exact rows:
 \[
 \raisebox{5.5pc}{$\xymatrix@R-1.4pc{
  0\ar[r] & 0\dplus \ar[r] & G_{\beta'\circ \Phi}^0\dplus\ar[r] & G_{\beta'}^0\dplus\ar[r] & 0\\
  0\ar[r] & G_\Phi^1\ar[r]\ar@{=}[dd] & G_{\beta'\circ \Phi}^1\ar[r]\ar[dd] & G_{\beta'}^1\ar[r]\ar[dd] & 0\\ \\
  0\ar[r] & G_\Phi\ar[r] & D(L^{*})\ar[r] & D(L'^{*})\ar[r] & 0
 }$}\qedhere\]
\end{proof}

\begin{lemma}\label{lem:ses=>ses-of-G_betas}
 Suppose we have the commutative diagram of finitely generated abelian groups below on the left, in which the rows are exact. Suppose $\ker \beta_0$, $\ker \beta$, and $\ker \beta'$ are free.
 \[\xymatrix{
  0\ar[r] & L_0\ar[r]\ar[d]_{\beta_0} & L\ar[r]\ar[d]_\beta & L'\ar[r]\ar[d]^{\beta'} & 0\\
  0\ar[r] & N_0\ar[r] & N\ar[r] & N'\ar[r] & 0
 }\qquad\xymatrix@C-.8pc{
  0\ar[r] & G_{\beta_0}\ar[r]\ar[d] & G_\beta\ar[r]\ar[d] & G_{\beta'}\ar[r]\ar[d] & 0\\
  0\ar[r] & D(L_0^{*})\ar[r] & D(L^{*})\ar[r] & D(L'^{*})\ar[r] & 0
 }\]
 If $\cok\beta_0$ is finite, then the rows of the diagram on the right are exact.
\end{lemma}
\begin{proof}
 We are given a short exact sequence $0\to C(\beta_0)\to C(\beta)\to C(\beta')\to 0$, which gives us an exact triangle in the derived category. The dual is then again an exact triangle, so we get a long exact sequence of homology groups
 \[
  0\to D(G_{\beta'}^{0})\to D(G_\beta^{0})\to \underbrace{D(G_{\beta_0}^{0})}_0\to D(G_{\beta'}^{1})\to D(G_\beta^{1})\to D(G_{\beta_0}^{1})\to 0.
 \]
 Since $\cok\beta_0$ is finite, we have that $D(G_{\beta_0}^{0})=0$, so we get the following diagram with exact rows:
 \[
 \raisebox{5.5pc}{$\xymatrix@R-1.4pc{
  0\ar[r] & 0\dplus \ar[r] & G_\beta^0\dplus\ar[r] & G_{\beta'}^0\dplus\ar[r] & 0\\
  0\ar[r] & G_{\beta_0}^1\ar[r]\ar[dd] & G_\beta^1\ar[r]\ar[dd] & G_{\beta'}^1\ar[r]\ar[dd] & 0\\ \\
  0\ar[r] & D(L_0^{*})\ar[r] & D(L^{*})\ar[r] & D(L'^{*})\ar[r] & 0
 }$}\qedhere\]
\end{proof}

\begin{lemma}\label{lem:gen-stabilizer<->torsion}
  Suppose $\beta\colon L\to N$ is a homomorphism from a lattice to a finitely generated abelian group, and suppose $\cok\beta$ is finite. The map $G_{0\to N_\tor}\to G_\beta$ induced by the commutative square
  \[\xymatrix{
    0\ar[d]\ar[r] & L\ar[d]^\beta\\
    N_\tor \ar@{^(->}[r] & N
  }\]
  is the inclusion of the kernel of the map $G_\beta\to T_L$.
\end{lemma}
\begin{proof}
  Consider the diagram below on the left. By Lemma \ref{lem:ses=>ses-of-G_betas}, the diagram below on the right has exact rows.
   \[\xymatrix@C-.5pc{
   0\ar[r] & 0\ar[d]\ar[r] & L\ar[d]^\beta \ar@{=}[r] & L\ar[d]^{\beta'}\ar[r] & 0\\
   0\ar[r] & N_\tor\ar[r] & N\ar[r] & N/N_\tor\ar[r] & 0
  }\qquad
  \xymatrix@C-.8pc{
   0\ar[r] & G_{0\to N_\tor}\ar[r]\ar[d] & G_\beta\ar[d]\ar[r] & G_{\beta'} \ar[r]\ar@{^(->}[d] & 0\\
   0\ar[r] & 0\ar[r] & T_L\ar@{=}[r] & T_L\ar[r] & 0
  }\]
  Since $\beta'$ has finite cokernel, the map $G_{\beta'}=D(\cok(\beta'^*))\to T_L$ is an injection. The result follows by applying the snake lemma to the diagram on the right.
\end{proof}
}

{\section*{Appendix B: Toric Isomorphisms}
\renewcommand{\thesection}{B}
\refstepcounter{section}
\label{sec:non-uniqueness-of-fans}

As mentioned in Warning \ref{warn:non-canonical-presentations}, non-isomorphic stacky fans can give rise to isomorphic toric stacks. The goal of this appendix is to prove Theorem \ref{thm:isomorphisms}, which characterizes when a morphism of non-strict stacky fans gives rise to an isomorphism of non-strict toric stacks. The proof appears at the end of this section.

\begin{definition}\label{def:Phi^-1(sigma)}
 Suppose $\Phi\colon \Sigma\to \Sigma'$ is a morphism of fans and $\sigma'\in \Sigma'$. The \emph{pre-image} of $\sigma'$, $\Phi^{-1}(\sigma')$, is the subfan of $\Sigma$ consisting of cones whose images lie in $\sigma'$.
\end{definition}
\begin{remark}\label{rmk:meaning-of-preimage}
 Suppose $f\colon X_\Sigma\to X_{\Sigma'}$ is the morphism of toric varieties corresponding to the map of fans $\Sigma\to \Sigma'$. The cone $\sigma'\in \Sigma'$ corresponds to an affine open subscheme $U_{\sigma'}=\spec(k[\sigma'^\vee\cap N'])\subseteq X_{\Sigma'}$, the complement of the divisors corresponding to rays not on $\sigma'$. The key property of $\Phi^{-1}(\sigma')$ is that $X_\Sigma\times_{X_{\Sigma'}}U_{\sigma'}$ is naturally the open subvariety $X_{\Phi^{-1}(\sigma')}\subseteq X_\Sigma$.
\end{remark}

\begin{theorem}\label{thm:isomorphisms}
 Let $(\Phi,\phi)\colon (\Sigma,\beta\colon L\to N)\to (\Sigma',\beta'\colon L'\to N')$ be a morphism of non-strict stacky fans, where $\cok\beta$ and $\cok\beta'$ are finite. The induced morphism $\X_{\Sigma,\beta}\to \X_{\Sigma',\beta'}$ is an isomorphism if and only if the following conditions hold:
 \begin{enumerate}
  \item $\phi$ is an isomorphism,
  \item for every cone $\sigma'\in \Sigma'$, $\Phi^{-1}(\sigma')$ is a single cone (in the sense of Notation \ref{not:sigma-fan}), and
  \item for every cone $\sigma'\in \Sigma'$, $\Phi$ induces a isomorphism of monoids $\Phi^{-1}(\sigma')\cap L\to \sigma'\cap L'$.
 \end{enumerate}
\end{theorem}

\begin{notation}\label{not:sigma-fan}
 As a slight abuse of notation, we use the symbol $\sigma$ to denote a cone as well as the fan consisting of the cone $\sigma$ and all of its faces.
\end{notation}

\begin{definition}\label{def:cohomologically-affine}
 A non-strict toric stack $\X_{\Sigma,\beta}$ is \emph{cohomologically affine} if $X_\Sigma$ is affine (i.e.~if $\Sigma$ is a single cone $\sigma$).
\end{definition}

\begin{remark}\label{rmk:coh-affine-definitions-agree}
 More generally, an algebraic stack $\X$ is said to be cohomologically affine if the structure morphism $\X\to \spec k$ is cohomologically affine in the sense of Definition \ref{def:gms}. This is in agreement with Definition \ref{def:cohomologically-affine}. If a toric variety $X_{\sigma}$ is an affine toric variety and $G_\beta$ is an affine group, then $[X_\sigma/G_\beta]\to \spec k$ is cohomologically affine by \cite[Proposition 3.14]{Alper:good}.

 Conversely, suppose $\X_{\Sigma,\beta}=[X_\Sigma/G_\beta]\to \spec k$ is cohomologically affine. Since $X_{\Sigma}$ is affine over $\X_{\Sigma,\beta}$ (because it is a $G_\beta$-torsor), we have that $X_\Sigma\to \spec k$ is cohomologically affine, so $X_\Sigma$ is affine by Serre's criterion \cite[Corollary 5.2.2]{ega}. It follows that $\Sigma$ is a single cone.
\end{remark}

\begin{lemma}\label{lem:coh-affine-surjection}
 Let $(\Phi,\phi)\colon (\sigma,\beta)\to (\Sigma',\beta')$ be a morphism of non-strict stacky fans, with $\sigma$ a single cone. Suppose every torus orbit of $\X_{\Sigma',\beta'}$ is in the image of the induced morphism $\X_{(\Phi,\phi)}$ (e.g.~if $\X_{(\Phi,\phi)}$ is surjective). Then the map $\sigma\to \Sigma'$ is a surjection. In particular, $\Sigma'$ is a single cone, so $\X_{\Sigma',\beta'}$ is cohomologically affine.
\end{lemma}
\begin{proof}
 The cones of $\Sigma'$ correspond to torus orbits of $X_{\Sigma'}$ and $\X_{\Sigma',\beta'}$. We see that every torus orbit of $X_{\Sigma'}$ is in the image of $X_\sigma$. Thus, the relative interior of every cone of $\Sigma'$ contains the image of some face of $\sigma$. Therefore $\Phi(\sigma)$ intersects the relative interiors of all cones, and in particular all rays, of $\Sigma'$. Since $\Phi(\sigma)$ is a convex polyhedral cone, it follows that the $\Phi(\sigma)$ is the cone generated by the rays of $\Sigma'$. In particular, $\Sigma'\subseteq \Phi(\sigma)$.
\end{proof}
\begin{remark}
 The cohomological affineness condition in Lemma \ref{lem:coh-affine-surjection} cannot be removed. For example, let $\Sigma =
 \raisebox{-2pt}{\begin{tikzpicture}[scale=.5,thick]
   \draw[<->] (1,0) -- (0,0) -- (0,1);
   \draw[->] (0,0) -- (1,1);
 \end{tikzpicture}}$ and $\Sigma'=
 \raisebox{-2pt}{\begin{tikzpicture}[scale=.5,thick]
   \filldraw[draw=none,fill=lightgray] (0,0) rectangle (1,1);
   \draw[<->] (1,0) -- (0,0) -- (0,1);
 \end{tikzpicture}}$.
 Then $X_\Sigma$ is the blowup of $\AA^2$ at the origin, minus the two torus-invariant points of the exceptional divisor, and $X_{\Sigma'}$ is $\AA^2$. The natural map $X_\Sigma\to X_{\Sigma'}$ is surjective, but the map on fans is not.
 \end{remark}

\begin{definition}\label{def:pointed}
 An affine toric variety $X_\sigma$ is \emph{pointed} if it contains a torus-invariant point. Alternatively, $X_\sigma$ is pointed if $\sigma$ spans the ambient lattice $L$. If $(\sigma,\beta)$ is a non-strict stacky fan and $X_\sigma$ is pointed, we say that $\X_{\sigma,\beta}$ is a \emph{pointed non-strict toric stack}. In this case, we say that $(\sigma,\beta)$ is a \emph{pointed non-strict stacky fan}.
\end{definition}
Note that a pointed affine toric variety has a \emph{unique} torus-invariant point.

\begin{remark}\label{rmk:canonical-pointed-subvariety}
 Note some immediate consequences of the equivalence of categories between toric varieties and fans. Any affine toric variety $X_\sigma$ has a canonical pointed toric subvariety. Explicitly, let $L_\sigma\subseteq L$ be the sublattice spanned by $\sigma$. Then $X_{\sigma,L_\sigma}$ is pointed, and the inclusion $L_\sigma\to L$ induces a toric closed immersion $X_{\sigma,L_\sigma}\to X_{\sigma,L}$. We have that $X_{\sigma,L}$ is (non-canonically) isomorphic to the product of $X_{\sigma,L_\sigma}$ and the torus $T_{L/L_\sigma}$.
\end{remark}

\begin{remark}\label{rmk:splitting-off-torus}
 Similarly, if $(\sigma,\beta\colon L\to N)$ is a non-strict stacky fan, with $\sigma$ a single cone, there is a canonical morphism from a pointed toric stack. Let $L_\sigma\subseteq L$ be the saturated sublattice generated by $\sigma$, let $N_\sigma=\sat_N(\beta(L_\sigma))$, and let $\beta_\sigma\colon L_\sigma\to N_\sigma$ be the morphism induced by $\beta$. Then we have a morphism of non-strict stacky fans $(\sigma,\beta_\sigma)\to (\sigma,\beta)$. We see that $\X_{\sigma,\beta}$ is (non-canonically) isomorphic to the product of the pointed non-strict toric stack $\X_{\sigma,\beta_\sigma}$ and the stacky torus $\X_{\ast,L/L_\sigma\to N/N_\sigma}$, where $\ast$ is the trivial fan.
\end{remark}
\begin{remark}\label{rmk:pointed-factor-universal-property}
 Note that $\X_{\sigma,\beta_\sigma}$, as defined in Remark \ref{rmk:splitting-off-torus}, has the following property. Any toric morphism to $\X_{\sigma,\beta}$ from a pointed toric stack factors uniquely through $\X_{\sigma,\beta_\sigma}$. This follows immediately from the fact that any toric morphism from a pointed toric stack to a stacky torus must be trivial.
\end{remark}

\begin{lemma}\label{lem:toric-torsor-over-pointed}
 Let $X_\sigma$ be a pointed affine toric variety, and let $Y\to X_\sigma$ be a toric morphism from a toric variety making $Y$ into a $G$-torsor over $X_\sigma$ for some group $G$. Then $Y\to X_\sigma$ is a trivial torsor and $G$ is a torus. In particular, $Y\cong G\times X_\sigma$ as a toric variety, so there is a canonical toric section $X_\sigma\to Y$.
\end{lemma}
\begin{proof}
 We have that $G$ is the kernel of the homomorphism of tori induced by $Y\to X_\sigma$, so it is diagonalizable. We decompose $G$ as a product of a finite group $G_0$ and a torus $\GG_m^r$.

 We then have that $Y_0 = Y/\GG_m^r$ is a toric variety which is a $G_0$-torsor over $X_\sigma$. The fiber over the torus-invariant point of $X_\sigma$ is then a torus invariant finite subset of $Y_0$. A torus has no finite-index subgroups, so any finite torus-invariant subset of a toric variety must consist of \emph{fixed points} of the torus action. On the other hand, $Y_0$ is affine over $X_\sigma$, so it is affine, so it contains at most one torus fixed point. Therefore, $G_0$ must be trivial, so $G\cong\GG_m^r$.

 We have that $G$-torsors over $X_\sigma$ are parametrized by $H^1(X_\sigma,G)\cong H^1(X_\sigma,\GG_m)^r \cong \pic(X_\sigma)^r$. By \cite[Proposition 4.2.2]{cls}, we have that $\pic(X_\sigma)=0$, so all $\GG_m^r$-torsors on $X_\sigma$ are trivial. It follows that $Y=\GG_m^r\times X_\sigma$. By Remark \ref{rmk:canonical-pointed-subvariety} there is a canonical toric section.
\end{proof}

\begin{corollary}\label{cor:morphisms-from-pointed}
 Let $\X_{\sigma,\beta\colon L\to N}$ be a pointed non-strict toric stack, where $\sigma$ spans the lattice $L$. Let $f\colon \X_{\sigma,\beta}\to \X_{\Sigma',\beta'}$ be a homomorphism of toric stacks. Then $f$ is induced by a morphism of stacky fans $(\sigma,\beta)\to (\Sigma',\beta')$.
\end{corollary}
\begin{proof}
 Following the proof (and notation) of Theorem \ref{thm:morphisms-come-from-fans}, we see that there is toric variety $Y_0$ with a toric morphism $Y_0\to X_\sigma$ making $Y_0$ into a $G_\Phi$-torsor over $X_\sigma$. By Lemma \ref{lem:toric-torsor-over-pointed}, there is a canonical toric section $s$. This toric section induces a section of the morphism of stacky fans $(\Phi,\id_N)\colon (\Sigma_0,\beta_0)\to (\sigma,\beta)$.
 \[\xymatrix{
  L \ar[d]_\beta \ar@/_/[r] \POS p+(0,.7) *+{\sigma}="s" & L_0 \ar[l]_{\Phi}\ar[r]\ar[d]^{\beta_0} \POS p+(0,.7) *+{\Sigma_0}="s0" \ar "s" \ar@{<-}@/_/_{s} "s" & L'\ar[d]^{\beta'} \POS p+(0,.7) *+{\Sigma'} \ar@{<-} "s0"\\
  N\ar@{=}[r] & N\ar[r]^\phi & N'
 }\]
 The composition $(\sigma,\beta)\to (\Sigma',\beta')$ then induces the morphism $f$.
\end{proof}

\begin{corollary}\label{cor:points-in-fan<->maps-from-A^1}
 Let $(\Sigma,\beta\colon L\to N)$ be a non-strict stacky fan. There is a natural bijection between toric morphisms $\AA^1\to \X_{\Sigma,\beta}$ and elements of $L$ which lie in the support of $\Sigma$.
\end{corollary}
\begin{proof}
 This follows immediately from Corollary \ref{cor:morphisms-from-pointed} and the usual description of the fan of $\AA^1$, namely $(\raisebox{3pt}{\tikz[scale=.5] \draw[thick,->] (0,0) -- (1,0);}, \id\colon \ZZ\to \ZZ)$.
\end{proof}

\begin{lemma}\label{lem:iso->iso-of-monoids}
 Let $(\Phi,\phi)\colon (\sigma,\beta\colon L\to N)\to (\sigma',\beta'\colon L'\to N')$ be a morphism of non-strict stacky fans, with $\sigma$ and $\sigma'$ single cones. Suppose the induced morphism $\X_{(\Phi,\phi)}$ is an isomorphism. Then $\Phi$ induces an isomorphism of monoids $\sigma\cap L\to \sigma'\cap L'$.
\end{lemma}
\begin{proof}
 By Corollary \ref{cor:points-in-fan<->maps-from-A^1}, elements of the monoid $\sigma\cap L$ are in bijection with toric morphisms from $\AA^1$ to $\X_{\sigma,\beta}$. The isomorphism $\X_{(\Phi,\phi)}$ induces a bijection of these sets. On the other hand, the induced morphism is a morphism of monoids.
\end{proof}


\begin{lemma}\label{lem:adding-rank-to-L}
 Let $\Sigma$ be a fan on a lattice $L$ and let $\beta\colon L\to N$ be a finite-cokernel morphism to a finitely generated abelian group. Let $L_0$ be a lattice and $\beta_0\colon L_0\to N$ any homomorphism. Let $\Sigma\times 0$ be the fan $\Sigma$ regarded as a fan on $L\oplus L_0$, supported entirely on $L$. Then the morphism of non-strict stacky fans $(\Sigma,\beta\colon L\to N)\to (\Sigma\times 0,\beta\oplus \beta_0\colon L\oplus L_0\to N)$ induces an isomorphism $\X_{\Sigma,\beta}\to \X_{\Sigma\times 0,\beta\oplus \beta'}$.
\end{lemma}
\begin{proof}
 Since $\Sigma\times 0$ is the product of $\Sigma$ on $L$ and the trivial fan on $L_0$, we have that $X_{\Sigma\times 0} = X_\Sigma\times T_{L_0}$. The diagram on the left has exact rows.
 \[\xymatrix@C-.5pc{
  0\ar[r] & L\ar[d]_\beta\ar[r] & L\oplus L_0\ar[d]^{\beta\oplus \beta_0} \ar[r] & L_0\ar[d]\ar[r] & 0\\
  0\ar[r] & N\ar@{=}[r] & N\ar[r] & 0\ar[r] & 0
 }\qquad
 \xymatrix@C-.8pc{
  0\ar[r] & G_\beta\ar[r]\ar[d] & G_{\beta\oplus \beta'}\ar[d]\ar[r] & G_{L_0\to 0} \ar[r]\ar[d]^\wr & 0\\
  0\ar[r] & T_L\ar[r] & T_L\oplus T_{L_0}\ar[r] & T_{L_0}\ar[r] & 0
 }\]
 Since $\cok\beta$ is finite, Lemma \ref{lem:ses=>ses-of-G_betas} implies that the second diagram has exact rows. We see that the induced morphism is then
 \begin{align*}
  \X_{\Sigma,\beta}=[X_\Sigma/G_\beta]\xrightarrow\sim \bigl[[X_{\Sigma\times 0}/T_{L_0}]/G_\beta\bigr] &=\bigl[[X_{\Sigma\times 0}/G_{L_0\to 0}]/G_\beta\bigr]\\
  & = [X_{\Sigma\times 0}/G_{\beta\oplus\beta_0}] = \X_{\Sigma\times 0,\beta\oplus\beta_0}.\qedhere
 \end{align*}
\end{proof}

\begin{remark}
 The condition that $\cok\beta$ is finite is necessary in the above argument. As a simple counterexample where $\cok\beta$ is not finite, consider the morphism of non-strict stacky fans $(0,\id)\colon (\Sigma,0\to \ZZ)\to (\Sigma,\id\colon \ZZ\to \ZZ)$, where $\Sigma$ is the trivial fan. This induces the morphism $B\GG_m\to \GG_m$.
\end{remark}

\begin{lemma}\label{lem:adding-to-L-closely}
 Let $\sigma$ be a cone on $L$. Suppose $(\Phi,\phi)\colon (\sigma,\beta\colon L\to N)\to (\Phi(\sigma),\beta'\colon L'\to N')$ is a morphism of non-strict stacky fans so that $\phi$ is an isomorphism, $\cok\Phi$ is finite, and so that $\Phi$ induces an isomorphism of monoids $(\sigma\cap L)\to (\Phi(\sigma)\cap L')$. Then the induced morphism $\X_{\sigma,\beta}\to \X_{\Phi(\sigma),\beta'}$ is an isomorphism.
\end{lemma}
\begin{proof}
 Let $L_\sigma\subseteq L$ be the sublattice generated by $\sigma\cap L$, and let $L_1$ be a direct complement. Let $X_\sigma$ denote the pointed affine toric variety corresponding to the cone $\sigma$, regarded as a fan on $L_\sigma$. The fan $\sigma$ (on $L$) is the product $\sigma\times 0$ on $L_\sigma\times L_1$, so $\X_{\sigma,\beta} = [(X_\sigma\times T_{L_1})/G_{\beta}]$.

 By assumption, the sublattice of $L'$ generated by $\Phi(\sigma)\cap L'$ is isomorphic to $L_\sigma$, and the isomorphism identifies $\sigma$ with $\Phi(\sigma)$. Let $L_1'$ be a direct complement to $L_\sigma$ in $L'$. As above, we have that $\X_{\Phi(\sigma),\beta'} = [(X_\sigma\times T_{L_1'})/G_{\beta'}]$.

 Note that $T_L=T_{L_\sigma}\oplus T_{L_1}$ and $T_{L'}=T_{L_\sigma}\oplus T_{L_1'}$. Since $\cok\Phi$ is finite, Lemma \ref{lem:triangle=>ses-of-G_betas} tells us that the following diagram has exact rows:
 \[\xymatrix@R-1.7pc{
  0\ar[r] & G_\Phi \ar[r]\ar@{=}[dd] & G_\beta \ar[r]\ar[dd] & G_{\beta'}\ar[r]\ar[dd] & 0\\ \\
  0\ar[r] & G_\Phi\dplus \ar[r] & T_{L_1}\dplus \ar[r] & T_{L_1'}\dplus \ar[r] & 0\\
  0\ar[r] & 0\ar[r] & T_{L_\sigma}\ar@{=}[r] & T_{L_\sigma}\ar[r] & 0
 }\]
 So we see that the toric morphism induced by $(\Phi,\phi)$ is the isomorphism
 \[
   \X_{\sigma,\beta} = [(X_\sigma\times T_{L_1})/G_\beta]=\bigl[[(X_\sigma\times T_{L_1})/G_\Phi]/G_{\beta'}\bigr] = [(X_\sigma\times T_{L_1'})/G_{\beta'}]=\X_{\Phi(\sigma),\beta'}. \qedhere
 \]
\end{proof}

\begin{lemma}\label{lem:product-of-isos-an-iso}
 For $i=0,1$, let $(\Phi_i,\phi_i)$ be a morphism of non-strict stacky fans. Then $(\Phi_0\times \Phi_1,\phi_0\times \phi_1)$ induces an isomorphism of non-strict stacky toric stacks if and only if each $(\Phi_i,\phi_i)$ does.
\end{lemma}
\begin{proof}
 This is an immediate corollary of Proposition \ref{prop:product-of-property-P}.
\end{proof}

Lemmas \ref{lem:adding-rank-to-L} and \ref{lem:adding-to-L-closely} can be combined and extended.
\begin{proposition}\label{prop:isos-of-affines}
 Let $(\Phi,\phi)\colon (\sigma,\beta\colon L\to N)\to (\sigma',\beta'\colon L'\to N')$ be a morphism of non-strict stacky fans, with $\sigma$ and $\sigma'$ single cones. Suppose
 \begin{enumerate}
  \item \label{condition1} $\phi$ is an isomorphism,
  \item $\Phi$ induces an isomorphism of the monoids $(\sigma\cap L)$ and $(\sigma'\cap L')$, and
  \item \label{condition3} $\phi\bigl(\sat_N\beta(L)\bigr) = \sat_{N'}\beta'(L')$.
 \end{enumerate}
 Then the induced morphism $\X_{\sigma,\beta}\to \X_{\sigma',\beta'}$ is an isomorphism.
\end{proposition}
\begin{proof}
 First we reduce to the case when $\cok\beta$ is finite. Let $N_0$ be a direct complement to $\sat_N \beta(L)$. Then we see that $(\Sigma,\beta\colon L\to N)$ is the product of $\bigl(\Sigma,\beta_1\colon L\to \sat_N \beta(L)\bigr)$ and $(\ast,\beta_0\colon 0\to N_0)$, where $\ast$ is the trivial fan (which contains only the zero cone). Condition (\ref{condition3}) implies that $\phi(N_0)$ is a direct complement to $\sat_{N'}\beta'(L')$, so the same argument shows that $(\Sigma',\beta'\colon L'\to N')$ is the product of $\bigl(\Sigma',\beta'_1\colon L'\to \sat_{N'} \beta'(L')\bigr)$ and $\bigl(\ast,\beta_0'\colon 0\to \phi(N_0)\bigr)$. We see that $(\Phi,\phi)$ is a product of $\bigl(\Phi\colon L\to L',\phi_1\colon \sat_N(\beta(L))\to\sat_{N'}(\beta'(L'))\bigr)$ and $(0,\phi|_{N_0})$. The latter is an isomorphism by conditions (\ref{condition1}) and (\ref{condition3}), so by Lemma \ref{lem:product-of-isos-an-iso}, we have reduced to the case where $N=\sat_N \beta(L)$ (i.e.~where $\cok\beta$ is finite).

 Let $L_\sigma$ be the sublattice of $L$ generated by $\sigma\cap L$. Let $N_1$ be the free part of $\beta(L)$, and choose a splitting $s\colon N_1\to L$. Let $L_1=L_\sigma+s(N_1)$ in $L$.

 Applying Lemmas \ref{lem:adding-to-L-closely} and \ref{lem:adding-rank-to-L} in succession, we see that the morphism induced by the composition $(\sigma,\beta|_{L_1})\to (\sigma,\beta|_{\sat_L L_1})\to (\sigma,\beta)$ is an isomorphism.

 Note that $\Phi|_{L_1}$ has no kernel, so we may identify $\Phi(L_1)$ with $L_1$. Then the same argument shows that $(\sigma,\beta|_{L_1})\to (\sigma',\beta')$ induces an isomorphism.

 We then have a factorization of morphisms of non-strict stacky fans $(\sigma,\beta|_{L_1})\to (\sigma,\beta)\xrightarrow{(\Phi,\phi)} (\sigma',\beta')$. Since the first morphism induces an isomorphism and the composite induces an isomorphism, it follows that $(\Phi,\phi)$ induces an isomorphism.
\end{proof}

We conclude this section with the proof of Theorem \ref{thm:isomorphisms}.

\begin{proof}[Proof of Theorem \ref{thm:isomorphisms}]
 Suppose the conditions hold. To show $\X_{\Sigma,\beta}\to \X_{\Sigma',\beta'}$ is an isomorphism, it suffices to look locally on the base, so we may assume $\Sigma'$ is a single cone $\sigma'$. The result then follows from Proposition \ref{prop:isos-of-affines}.

 Conversely, suppose $(\Phi,\phi)$ induces an isomorphism. By Lemma \ref{lem:gen-stabilizer<->torsion}, the map of generic stabilizers is $G_{0\to N_\tor}\to G_{0\to N'_\tor}$. Since this is an isomorphism, it follows that $\phi$ restricts to an isomorphism $\phi_\tor\colon N_\tor\to N'_\tor$. The quotient map $\bar\phi\colon N/N_\tor\to N'/N'_\tor$ induced by $\phi$ is the map of 1-parameter subgroups of the stacky tori of $\X_{\Sigma,\beta}$ and $\X_{\Sigma',\beta'}$, respectively, so it is an isomorphism. Since $\phi_\tor$ and $\bar\phi$ are isomorphisms, it follows by the 5-lemma that $\phi$ is an isomorphism.

 Since isomorphisms are affine, Remark \ref{rmk:meaning-of-preimage} shows that $\Phi^{-1}(\sigma')$ is a single cone for each $\sigma'\in\Sigma'$. Finally, Lemma \ref{lem:iso->iso-of-monoids} shows that $\Phi$ induces isomorphisms of monoids.
\end{proof}

\begin{remark}\label{rmk:smooth-toric-stacks-are-fantastacks}
 Suppose $X_\Sigma$ is a smooth toric variety and $\Phi\colon (\ttilde\Sigma,\ZZ^n)\to (\Sigma,L)$ is its Cox construction \cite[\S 5.1]{cls}. Then the induced morphism $\X_{\ttilde\Sigma,\Phi}\to \X_{\Sigma,\id_L}=X_\Sigma$ is an isomorphism. In particular, any smooth toric variety can be expressed as a quotient of a $\GG_m^n$-invariant open subvariety of $\AA^n$ by a subgroup of $\GG_m^n$. Toric stacks of this form are called \emph{fantastacks} (see Section \ref{sec:fantastacks} and Example \ref{eg:fantastack-smooth-toric-varieties}).
\end{remark}

}



\end{document}